\documentclass[11pt]{amsart}

\usepackage{cmap}
\usepackage[T2A]{fontenc}
\usepackage[utf8]{inputenc}
\usepackage[english]{babel}
\usepackage{amsfonts,amssymb}
\usepackage{hyperref}
\usepackage{stmaryrd}
\usepackage[left=2.0cm,right=2.0cm,top=2.0cm,bottom=2.0cm]{geometry}

\usepackage[all,cmtip]{xy}
\usepackage{capt-of} 

\newcommand{\mytitle}{Overgroups of exterior powers of an elementary group.\\\MakeUppercase{\romannumeral 1.} Levels and normalizers}
\title{\mytitle}

\author{Roman~Lubkov}
\thanks{The constructions from the proof of Theorems 1 and 2 (Subsections \ref{square}--\ref{normgeneral}) are done by the first author and supported by the Russian Science Foundation, grant №17-11-01261.}
\address[Roman~Lubkov]{Department of Mathematics and Mechanics, St.~Petersburg State University}
\email{RomanLubkov@yandex.ru}

\author{Ilia~Nekrasov}
\thanks{The constructions from the proof of Theorem 3 (Subsections \ref{stabgood}--\ref{stabbad}) are done by the second author and supported by the Russian Science Foundation, grant №16-11-10200. Also the second author is thankful to ``Native towns'', a social investment program of PJSC ``Gazprom Neft''.}
\address[Ilia~Nekrasov]{Chebyshev Laboratory, St.~Petersburg State University}
\email{geometr.nekrasov@yandex.ru}


\date{}
\DeclareMathOperator{\E}{E}
\DeclareMathOperator{\M}{M}
\DeclareMathOperator{\CC}{C}

\DeclareMathOperator{\Sp}{Sp}
\DeclareMathOperator{\GSp}{GSp}
\DeclareMathOperator{\EO}{EO}
\DeclareMathOperator{\SO}{SO}
\DeclareMathOperator{\GO}{GO}
\DeclareMathOperator{\SL}{SL}
\DeclareMathOperator{\GL}{GL}
\DeclareMathOperator{\sign}{sgn}
\DeclareMathOperator{\Tran}{Tran}
\DeclareMathOperator{\Gr}{Gr}
\DeclareMathOperator{\Sym}{Sym}
\DeclareMathOperator{\Stabi}{Stab}
\DeclareMathOperator{\height}{ht}
\DeclareMathOperator{\Ar}{Ar}

\makeatletter
\renewcommand{\trianglelefteq}{\trianglelefteqslant}
\renewcommand{\leq}{\leqslant}
\renewcommand{\geq}{\geqslant}
\makeatother

\newcommand{\bigwedgem}[1]{\mathord{\raisebox{2pt}
{\hbox{$\scriptstyle{\bigwedge^{\!#1}}$}}}}
\newcommand\blank{\mathord{\hbox to 1.5ex{\hrulefill}}\,}

\theoremstyle{plain}
\newtheorem{theorem}{Theorem}
\newtheorem*{generalproblem}{General Problem}
\newtheorem{thm}{Theorem}
\newtheorem{proposition}{Proposition}
\newtheorem{prop}{Proposition}
\newtheorem{lemma}{Lemma}
\newtheorem{corollary}{Corollary}
\theoremstyle{remark}
\newtheorem*{remark}{Remark}

\newcounter{MYc}

\begin{document}

\begin{abstract}
In the present paper, we prove the first part in the \textit{standard description} of groups $H$ lying between $\bigwedgem{m}\E(n,R)$ and $\GL_{\binom{n}{m}}(R)$. We study the structure of the group $\bigwedgem{m}\E(n,R)$ and its relative analog $\E\left(\binom{n}{m},R,A\right)$.  In the considering case $n \geq 3m$, the description is explained by the classical notion of \textit{level}: for every such $H$ we find unique ideal $A$ of the ring $R$. 

Motivated by the problem, we prove the coincidence of the following groups: normalizer of $\bigwedgem{m}\E(n,R)$, normalizer of $\bigwedgem{m}\SL_{n}(R)$, transporter of $\bigwedgem{m}\E(n,R)$ into $\bigwedgem{m}\SL_{n}(R)$, and a group $\bigwedgem{m}\GL_{n}(R)$. This result mainly follows from the found explicit equations for the algebraic group scheme $\bigwedgem{m}\GL_n(\blank)$.
\end{abstract}
\maketitle

\section*{Introduction}\label{fundrepres}

The present paper is devoted to the solution of the following general problem.

\begin{generalproblem}
Let $R$ be an arbitrary commutative associative ring with $1$ and invertible $2$, and let $\Phi$ be a reduced irreducible root system. For the Chevalley--Demazure scheme $G(\Phi, \blank) \leq \GL_{M}(\blank)$ and its arbitrary representation $\rho:\;G(\Phi, \blank) \longrightarrow \GL_{N}(\blank)$ describe all overgroups $H$ of the corresponding elementary group $\E_{G}(\Phi, \blank)$ in the representation $\rho$:
$$\E_{G, \rho}(\Phi, R ) \leq H \leq \GL_{N}(R).$$
\end{generalproblem}

The conjectural answer, \textit{standard overgroup description}, in a general case is formulated as follows. 
For any overgroup $H$ of the elementary group there exists a net of ideals $\mathbb{A}$ of the ring $R$ such that
$$\E_{G,\rho}(\Phi,R)\cdot\E(N,R,\mathbb{A})\leq H \leq N_{\GL_{N}(R)}\big(\E_{G,\rho}(\Phi, R)\cdot\E(N,R,\mathbb{A})\big),\eqno{(G_\mathbb{A})}$$
where $\E(N,R,\mathbb{A})$ is a \textit{relative elementary subgroup} for a net $\mathbb{A}$.

In the special case of a trivial net, i.\,e., $\mathbb{A} = \{A\}$, overgroups $H$ of the group $\E_{G}(\Phi,R)$ can be parametrized by the ideals $A$ of the ring $R$:
$$\E_{G,\rho}(\Phi,R)\cdot\E(N,R,A)\leq H \leq N_{\GL_{N}(R)}\big(\E_{G,\rho}(\Phi, R)\cdot\E(N,R,A)\big),\eqno{(G_A)}$$
where $\E(N,R,A)$ is a \textit{relative elementary subgroup} equals $\E(N, A)^{\E(N,R)}$. 

Moreover, the group $\mathrm{EE}_{G,\rho}(\Phi, R)=\E_{G,\rho}(\Phi,R)\cdot\E(N,R,A)$ is conjecturally perfect. And at the same time, the normalizer of the right-hand side is congruence subgroup for the corresponding level $A$.

The essential step in the proof of the right-hand side inclusion in $(G_A)$ is an alternative description of the normalizer $N_{\GL_{N}( \blank)}\big(\E_{G,\rho}(\Phi, \blank)\big)$. This group conjecturally coincides with a transporter $\Tran_{\GL_{N}(\blank)}(\E_{G,\rho}(\Phi, \blank), G_{\rho}(\Phi, \blank))$, a normalizer $N_{\GL_{N}}(G_{\rho}(\Phi, \blank))$ and the extended Chevalley--Demazure scheme $\overline{G}_{\rho}(\Phi, \blank)$.

The problem was solved in the following cases:
\begin{itemize}
\item for symplectic, orthogonal, and unitary groups by Nikolai Vavilov and Victor Petrov in \cite{VP-EOeven,VP-EOodd,VP-Ep,PetOvergr} and by You Hong, Xing Tao Wang and Cheng Shao Hong in  \cite{HongYou,YouHongOvergrClass,YouHongOvergrBanach,XingCheng};

\item for exceptional groups of Lie type $E_6,\;E_7,\;E_8,\; F_4$ by Nikolai Vavilov and Alexander Luzgarev in \cite{LuzE6E7-GL,LuzF4-E6,VavLuzgE6,VavLuzgE7};

\item for tensor product of general linear groups $\GL_m\otimes\GL_n$ the problem was partially solved by Alexey Ananievskii, Nikolai Vavilov, and Sergey Sinchuk in the joint paper \cite{AnaVavSinI}.
\end{itemize}   

We highly recommend the surveys \cite{VavSubgroup90,VavSbgs,VavStepSurvey}, which contain necessary preliminaries, complete history, and known results of the problem.

In the present paper, we consider the case of the $m$-th fundamental representation of a simply connected group of type $A_{n-1}$, i.\,e., the scheme $G_{\rho}(\Phi,\blank)$ equals the [Zariski] closure of the affine group scheme $\SL_{n}(\blank)$ in the representation with the highest weight $\varpi_{m}$. In our case, corresponding extended Chevalley group scheme coincides with the $m$-th fundamental representation of the general linear group scheme $\GL_{n}(\blank)$.

We deal only with a case of a trivial net, i.\,e., a unique parameterizing ideal $A$. As shown below (Propositions \ref{LevelsEqualGeneral} and \ref{InversInclus}), this restricts our consideration to the case $n\geq 3m$. In this case, the general answer has the following form. Let $H$ be a subgroup in $\GL_{N}(R)$ containing $\bigwedgem{m}\E(n,R)$. Then there exists a unique ideal $A\trianglelefteq R$ such that
$$\bigwedgem{m}\E(n,R)\cdot\E(N,R,A)\leq H \leq N_{\GL_{N}(R)}\big(\bigwedgem{m}\E(n,R)\cdot\E(N,R,A)\big).\eqno{(*)}$$

\medskip

The present paper is organized as follows. In next Section~\ref{main_results} we formulate all main results of the paper with only necessary major comments. In Section~\ref{princnot} we set notation that is used in the rest of the paper. Last lengthy Section~\ref{proofs} contains all complete proofs. In this section we firstly present the level computation for the particular case of the exterior square (Subsection~\ref{square}), and then develop the technique for a general exterior power (Subsections~\ref{general}--\ref{levelgeneral}). Subsections~\ref{stabgood}--\ref{stabbad} contain the presentation of the scheme $\bigwedgem{m}\GL_{n}(\blank)$ as a stabilizer. Also, the paper contains an Appendix devoted to the translation of the former technique for computing explicit equations defining group schemes into Representation Theory language.    

\textbf{Acknowledgment.} We would like to express our sincere gratitude to our scientific adviser Nikolai Vavilov for formulating the problem and constant support, without which this paper would never have been written. Also, authors are grateful to Alexei Stepanov for careful reading of our original manuscript and for numerous remarks and corrections.

\section{Main results}\label{main_results}  

Fundamental representations for the general linear group $\GL_{n}$, as well as for the special linear group $\SL_{n}$, are ones with the highest weights $\varpi_{m}$ for $m = 1, \dots, n$. The representation with the highest weight $\varpi_{n}$ degenerates in the case of $\SL_{n}$. The explicit description of these representations uses exterior powers of vector spaces.

In more details, for commutative ring $R$ by $\bigwedgem{m} R^{n}$ we denote $m$-th exterior power of the free module $R^n$. We can consider a natural transformation \textit{the exterior power} 
$$\bigwedgem{m}: \GL_{n} \rightarrow \GL_{\binom{n}{m}}$$
which action extends the action of the group of points $\GL_{n}(R)$ on the module $\bigwedgem{m} R^{n}$. 

The elementary group $\E(n, R)$ is a subgroup of the group of points $\GL_{n}(R)$, so its exterior power $\bigwedgem{m}\E(n,R)$ is well defined subgroup of the group of points $\bigwedgem{m}\GL_{n} (R)$. More user--friendly description of elements of the elementary group $\bigwedgem{m}\E(n,R)$ will be presented in the Subsections \ref{square} and \ref{general}.

Let $H$ be an arbitrary overgroup of the elementary group $\bigwedgem{m}\E(n,R)$:
$$\bigwedgem{m}\E(n,R) \leq H \leq \GL_{\binom{n}{m}}(R).$$ 
For any unequal weights $I, J \in \bigwedgem{m}[n]$, which are indices for matrix entries of elements from $\GL_{\binom{n}{m}}(R)$, by $A_{I,J}$ we denote the following set 
$$A_{I,J}:=\{\xi \in R \;|\; t_{I,J}(\xi) \in H \} \subseteq R.$$

It turns out that this sets are ideals indeed and coincide for all pairs of weights $I\neq J$. 

\begin{proposition}\label{InversInclus}
All sets $A_{I,J}$ coincide for $n \geq 3m$. 
\end{proposition}

In the unconsidered case $\frac{n}{3} \le m \le n$ the description of overgroups cannot be done by parametrization only by the single ideal. Moreover, as could be seen from further calculations we need from 1 to $m$ ideals for the complete parametrization of overgroups. In the general case, this [partially ordered] set of ideals forms the \textit{a net of ideals} (due to Zenon Borevich). Note that in the case of not-unique ideal parametrization, even the notion of relative elementary group is far more complex and depends on the Chevalley group (for instance, see~\cite{AnaVavSinI}), let alone formulations of the Main Theorems.

The set $A:= A_{I,J}$ is called \textit{a level of an overgroup} $H$. Description of the overgroups goes in the following way.

\begin{theorem}[Level Computation]\label{LevelForm}
Let $R$ be a commutative ring and $n$, $m$ be a natural numbers with a constraint $n\geq 3m$. For an arbitrary overgroup $H$ of the group $\bigwedgem{m}\E(n,R)$ there exists unique ideal $A$ of the ring $R$ such that 
$$\bigwedgem{m}\E(n,R)\cdot\E(N,R,A)\leq H.$$
Namely, if a transvection $t_{I,J}(\xi)$ belongs to the group $H$, then $\xi\in A$.
\end{theorem}

The left--hand side subgroup is denoted by $\E\bigwedgem{m}\E(n,R,A)$. We emphasize that this group is perfect (Lemma \ref{PerfectForm}).

Motivated by the expected relations $(*)$, we find an alternative description of the normalizer $N_{\GL_{N}(R)}(\E\bigwedgem{m}\E(n,R,A))$. For this we introduce the canonical projection $\rho_A: R\longrightarrow R/A$ mapping $\lambda\in R$ to $\bar{\lambda}=\lambda+A\in R/I$. Applying the projection to all entries of a matrix, we get the reduction homomorphism 
$$\begin{array}{rcl}
\rho_{A}:\GL_{n}(R)&\longrightarrow& \GL_{n}(R/A)\\
a &\mapsto& \overline{a}=(\overline{a}_{i,j})
\end{array}$$

Eventually, we have the following explicit \textit{congruence} description.

\begin{theorem}\label{LevelReduction} Let $n\geq 3m$.
For any ideal $A\trianglelefteq R$ we have
$$N_{\GL_{N}(R)}(\E\bigwedgem{m}\E(n,R,A))=\rho_{A}^{-1}\left(\bigwedgem{m}\GL_{n}(R/A)\right).$$
\end{theorem}

For formulating the last Theorem, we recall that for two arbitrary subgroups $E$ and $F$ of a group $G$ a \textit{transporter} of the group $E$ into the group $F$ is the following set
$$\Tran_G(E,F)=\{g\in G\;|\;E^g\leq F\},$$
In a special case of subgroups $E\leq F$ the last definition can be rephrased by using commutators:
$$\Tran_G(E,F)=\{g\in G\;|\;[g,E]\leq F\}.$$

The final Theorem describes a normalizer of the elementary group $\bigwedgem{m}\E(n,R)$ simultaneously as a transporter and as the group $\bigwedgem{m}\GL_{n}(R)$.

\begin{theorem}\label{normalizer}
Let $R$ be a commutative ring, and let $n\geq 4$. Then
$$N_{\GL_{N}(R)}(\bigwedgem{m}\E(n,R)) = N_{\GL_{N}(R)}(\bigwedgem{m}\SL_{n}(R)) = \Tran_{\GL_{N}(R)}(\bigwedgem{m}\E(n,R), \bigwedgem{m}\SL_{n}(R)) = \bigwedgem{m}\GL_{n}(R).$$
\end{theorem}

The most significant step in the proof of this Theorem is the inclusion 
$$\Tran_{\GL_{N}(R)}(\bigwedgem{m}\E(n,R), \bigwedgem{m}\SL(n,R) )\leq \bigwedgem{m}\GL_{n}(R).$$ 
Standard proof uses the invariant form defining the group $G_{\rho}(\Phi, R)$ as a stabilizer. However, in our case this forms exist only in some infrequent cases, unlike the classical ones of
\begin{itemize}
\item general orthogonal group  $\GO_{n}(\blank)$,
\item general symplectic group $\GSp_{n}(\blank)$,
\item groups of Lie type $\E_{6}$ 
\end{itemize}

For the remaining cases we consider (skew--symmetric or symmetric) polynomial ideals coming from coordinate ring of Grassmannians $\Gr(m,n)$ (Propositions \ref{StabAndGenerealPower}, \ref{mNotDividesn} for the case $n \neq 2m$ and Proposition \ref{StabPresExc} for the exceptional case $n = 2m$).

\section{Principal notation}\label{princnot}

Our notation is for the most part fairly standard in Chevalley group theory and coincides with the notation in \cite{ VP-EOeven, VP-EOodd,VP-Ep, StepVavDecomp}. We recall all necessary notion to read the present paper independently.

First, let $G$ be an arbitrary group. By a commutator of two elements we always understand \textit{the left-normed} commutator $[x,y]=xyx^{-1}y^{-1}$, where $x,y\in G$. Multiple commutators are also left-normed; in particular, $[x,y,z]=[[x,y],z]$. By ${}^xy=xyx^{-1}$ we denote \textit{the left conjugates} of $y$ by $x$. Similarly, by $y^x=x^{-1}yx$ we denote \textit{the right conjugates} of $y$ by $x$. 

In the sequel, we will use the Hall–Witt identity:
$$[x,y^{-1},z^{-1}]^x\cdot[z,x^{-1},y^{-1}]^z\cdot[y,z^{-1},x^{-1}]^y=e.$$

For a subset $X\subseteq G$, we denote by $\langle X\rangle$ the subgroup it generates. The notation $H\leq G$ means that $H$ is a subgroup in $G$, while the notation $H\trianglelefteq G$ means that $H$ is a normal subgroup in $G$. For $H\leq G$, we denote by $\langle X\rangle^H$ the smallest subgroup in $G$ containing $X$ and normalized by $H$. For two groups $F,H\leq G$, we denote by $[F,H]$ their mutual commutator: $[F,H]=\langle [f,g], \text{ where } f\in F, h\in H\rangle.$

Also, we need some elementary ring theory notation. Let $R$ be an arbitrary associative ring with 1. By default, it is assumed to be commutative. By an ideal $I$ of the ring $R$ we understand \textit{the two-sided ideal} and this is denoted by $I\trianglelefteq R$. As usual, let $R^{*}$ be the multiplicative group of the ring $R$. Let $M(m, n, R)$ be the $R$-bimodule of $(m \times n)$-matrices with entries in $R$, and let $M(n, R) = M(n, n, R)$ be the full matrix ring of degree $n$ over $R$. By $\GL_n(R)=\M(n,R)^{*}$ we denote the general linear group while $\SL_{n}(R)$ is the special linear group of degree $n$ over $R$. As usual, $a_{i,j}$ denotes the entry of a matrix $a$ at the position  $(i,j)$, where $1\leq i,j\leq n$. Further, $e$ denotes the identity matrix and $e_{i,j}$ denotes a standard matrix unit, i.\,e., the matrix that has 1 at the position $(i,j)$ and zeros elsewhere.

By $t_{i,j}(\xi)$ we denote an elementary transvection, i.\,e., a matrix of the form $t_{i,j}(\xi)=e+\xi e_{i,j}$, $1\leq i\neq j\leq n$, $\xi\in R$. In the sequel, we use (without any special reference) standard relations~\cite{StepVavDecomp} among elementary transvections such as
\begin{enumerate}
\item the additivity:
$$t_{i,j}(\xi)t_{i,j}(\zeta)=t_{i,j}(\xi+\zeta).$$
\item the Chevalley commutator formula:
$$[t_{i,j}(\xi),t_{h,k}(\zeta)]=
\begin{cases}
e,& \text{ if } j\neq h, i\neq k,\\
t_{i,k}(\xi\zeta),& \text{ if } j=h, i\neq k,\\
t_{h,j}(-\zeta\xi),& \text{ if } j\neq h, i=k.
\end{cases}$$
\end{enumerate}

Now, let $I$ be an ideal in $R$. Denote by $\E(n,I)$ the subgroup of $\GL_{n}(R)$ generated by all elementary transvections of level $I$:
$$\E(n,I)=\langle t_{i,j}(\xi), 1\leq i\neq j\leq n, \xi\in I\rangle.$$
In the most important case $I = R$, the group $\E(n,R)$ generated by all elementary transvections is called the \textit{(absolute)} elementary group:
$$\E(n,R)=\langle t_{i,j}(\xi), 1\leq i\neq j\leq n, \xi\in R\rangle.$$
It is well known (due to Andrei Suslin~\cite{SuslinSerreConj}) that elementary group is normal in the general linear group $\GL_{n}(R)$ for $n \geq 3$. The normality is crucial for the further considerations, so we suppose that $n \geq 3$.

In the sequel, the \textit{(relative)} elementary group $\E(n,R,I)$ is of great importance. Recall that the group $\E(n,R,I)$ is the normal closure of $\E(n,I)$ in $\E(n,R)$:
$$\E(n,R,I)=\langle t_{i,j}(\xi), 1\leq i\neq j\leq n, \xi\in I\rangle^{\E(n,R)}.$$

This group $\E(n,R,I)$ is normal in $\GL_{n}(R)$ if $n\geq 3$. This fact, first proved in \cite{SuslinSerreConj}, is cited as Suslin’s theorem. Moreover, if $n\geq 3$, then the group $\E(n,R,I)$ is generated by transvections of the form $z_{i,j}(\xi,\zeta)=t_{j,i}(\zeta)t_{i,j}(\xi)t_{j,i}(-\zeta)$, $1\leq i\neq j\leq n$, $\xi\in I$, $\zeta\in R$. This fact was proved by Vaserstein and Suslin~\cite{VasSusSerre} and, in context of Chevalley groups, by Tits~\cite{TitsCongruence}.

By $[n]$ we denote a set $\{1,2,\ldots, n\}$ and by $\bigwedgem{m}[n]$ we denote the exterior power of the set $[n]$. Elements of $\bigwedgem{m}[n]$ are unordered\footnote{In the sequel, we arrange them in ascending order.} subsets $I\subseteq [n]$ of cardinality $m$ without repeating entries:
$$\bigwedgem{m}[n] = \{ (i_{1}, i_{2}, \ldots , i_{m})\; |\; i_{j} \in [n], i_{j} \neq i_{l} \}.$$ 

For a matrix $a \in \GL_{n}(R)$ and for sets $I$, $J$ from $\bigwedgem{m}[n]$, define a minor $M_{I}^{J}(a)$ of the matrix $a$ as follows. $M_{I}^{J}(a)$ equals the determinant of a submatrix formed by rows from the set $I$ and columns from the set $J$. 

By $N$ we denote the binomial coefficient $\binom{n}{m}$. This number always equals the rank of the ambient group $\GL$. For this reason, in the sequel, we will not mention the dependence of the $N$ on $m$, exterior power.

\section{Proofs \& Computations}\label{proofs}

In this section we present all proofs of Main Theorems without skipping any technical details. Primarily we consider a case of the exterior square of the group scheme $\GL_{n}$. There are two reasons for this. Firstly, for $n=4$ Nikolai Vavilov and Victor Petrov completed the standard description of overgroups\footnote{The restriction of the exterior square map $\bigwedgem{2}: \GL_{4}(R) \longrightarrow \GL_{6}(R)$ to the group $\E(4,R)$ is an isomorphism onto the elementary orthogonal group $\EO(6,R)$~\cite{VP-EOeven}.}. Secondly, proofs of all statements in the arbitrary exterior power case are technically more complex analogues for the proofs of the present section. At the same time, simpler proofs present all basic ideas of the general case.

\subsection{Exterior square of elementary groups}\label{square}

Let $R^n$ be a right $R$--module with the basis $\{e_1,\ldots,e_n\}$. Consider the standard action of the group $\GL_{n}(R)$ on 
$R^n$. Define an exterior square of $R$--module as follows. The basis of this module is all exterior products $e_i\wedge e_j$, $1~\leq~i~\neq~j~\leq~n$ and $e_{i} \wedge e_{j} = -e_{j} \wedge e_{i}$. Denote the exterior square of $R^n$ by $\bigwedgem{2}(R^n)$. Now, we define an action of the group $\GL_{n}(R)$ on $\bigwedgem{2}(R^n)$. Firstly, we  define this action on basic elements by the rule
$$g(e_i\wedge e_j):=(ge_i)\wedge(ge_j) \text{ for any } g\in \GL_{n}(R) \text{ and } 1\leq i\neq j\leq n.$$
Secondly, we extend this action by linearity to the whole module $\bigwedgem{2}(R^n)$. Finally, using this action we define a subgroup $\bigwedgem{2}\big(\GL_{n}(R)\big)$ of the general linear group $\GL_{N}(R)$.

In other words, let us consider the Cauchy--Binet homomorphism
$$\bigwedgem{2}:\GL_{n}(R)\longrightarrow \GL_{N}(R),$$
taking each matrix $x\in \GL_{n}(R)$ to the matrix $\bigwedgem{2}(x)\in \GL_{N}(R)$. Elements of $\bigwedgem{2}(x)\in \GL_{N}(R)$ are all second order minors of the matrix $x$. Then the group $\bigwedgem{2}\big(\GL_{n}(R)\big)$ is an image of the general linear group under the Cauchy--Binet homomorphism. It is natural to index elements of the matrix $\bigwedgem{2}(x)$ by pair of elements of the set $\bigwedgem{2}[n]$:
$$\left(\bigwedgem{2}(x)\right)_{I,J}=\left(\bigwedgem{2}(x)\right)_{(i_1,i_2),(j_1,j_2)}=M_{i_1,i_2}^{j_1,j_2}(x) = x_{i_1, j_1}\cdot x_{i_2, j_2} - x_{i_1, j_2}\cdot x_{i_2, j_1}.$$

The following lemma is needed for the sequel; this fact is a corollary of Suslin's theorem.
\begin{lemma}\label{SuslinFor2}
The image of the elementary group is normal in the image of the general linear group under the exterior square homomorphism:
$$\bigwedgem{2}\left(\E(n,R)\right)\trianglelefteq \bigwedgem{2}(\GL_{n}(R)).$$
\end{lemma}

Consider the structure of the group $\bigwedgem{2}\E(n,R)$ in details. The following proposition can be obtained by the very definition of $\bigwedgem{2}\big(\GL_{n}(R)\big)$.
\begin{proposition} \label{ImageOfTransvFor2}
Let $t_{i,j}(\xi)$ be an elementary transvection. For $n\geq 3$ the transvection $\bigwedgem{2}t_{i,j}(\xi)$ can be presented as the following product:
$$\bigwedgem{2}t_{i,j}(\xi)=\prod\limits_{k=1}^{i-1} t_{ki,kj}(\xi)\,\cdot\prod\limits_{l=i+1}^{j-1}t_{il,lj}(-\xi)\,\cdot\prod\limits_{m=j+1}^n t_{im,jm}(\xi) \eqno(1)$$
for any $1\leq i<j\leq n$.
\end{proposition}

\begin{remark}
For $i>j$ similar equality holds:
$$\bigwedgem{2}t_{i,j}(\xi)=\prod\limits_{k=1}^{j-1} t_{ki,kj}(\xi)\,\cdot\prod\limits_{l=j+1}^{i-1}t_{li,jl}(-\xi)\,\cdot\prod\limits_{m=i+1}^n t_{im,jm}(\xi) \eqno(1')$$
\end{remark}

\begin{remark}
A commutator of any two transvections from the right-hand sides of formulas $(1)$ and $(1')$ equals~1. Therefore the commutator with the transvection $\bigwedgem{2}t_{i,j}(\xi)$ is equal to 1 as well.
\end{remark}
It follows from the proposition that $\bigwedgem{2}t_{i,j}(\xi)\in \E^{n-2}(N,R)$, where a set $\E^M(N,R)$ consists of products of $M$ or less elementary transvections.
 
Let $H$ be any overgroup of the exterior square of the elementary group $\bigwedgem{2}\E(n,R)$:
$$\bigwedgem{2}\E(n,R)\leq H \leq \GL_N(R).$$
Consider any different indices $I, J \in \bigwedgem{2}[n]$. By $A_{I,J}$ denote the set 
$$A_{I,J}:=\{\xi \in R \;|\; t_{I,J}(\xi) \in H \} \subseteq R.$$
Let $t_{I,J}(\xi)$ be any elementary transvection. Define \textit{the height} of $t_{I,J}(\xi)$ (generally, of the  pair $(I,J)$) as the cardinality of $I\cap J$:
$$\height(t_{I,J}(\xi))=\height(I,J)=|I\cap J|.$$
This combinatorial characteristic of transvections is useful in simplifying commutator calculations. 

The height splits up all sets $A_{I,J}$ into two classes: the one with $\height(I,J)=0$ and the one with $\height(I,J)=1$. In fact, these classes are equal for $n\geq 6$. The set $A := A_{I,J} $ is called  \textit{the level} of the overgroup $H$. Note that for $n=4$ the level is unique, that follows from~\cite{VP-EOeven}.

\begin{lemma}\label{IdealFor2}
Under a constraint $n\geq 6$ every set $A_{I,J}$ is an ideal in the ring $R$. Moreover, for any $I\neq J$ and $K\neq L$ the ideals $A_{I,J}$ and $A_{K,L}$ coincide.
\end{lemma}
\begin{proof} Complete proof is presented in Section~\ref{levelcomputation}, Proposition~\ref{LevelsEqualGeneral}. Here we only give calculations for the case $(n,m) = (4,2)$. The reason is that these calculations in a transparent way present the general idea.
\begin{enumerate}
\item Firstly, take any $\xi\in A_{12,34}$, i.\,e., $t_{12,34}(\xi)\in H$. Then 
$$[t_{12,34}(\xi),\bigwedgem{2}t_{4,2}(\zeta)]=t_{14,23}(-\xi\zeta^2)t_{14,34}(-\zeta\xi)t_{12,23}(-\xi\zeta)\in H.$$
It remains to provide this calculation with $-\zeta$ and to product two right-hand sides; then we obtain $t_{14,23}(-2\xi\zeta^2)\in H$. By the condition $2 \in R^{*}$, this means that $A_{12, 34} \subseteq A_{14, 23}$.  It follows that
$$A_{I,J} \subseteq A_{K,L}\text{ for } I\cup J = K \cup L = \{ 1234 \}.$$
\item Secondly, take any $\xi\in A_{12,34}$, then $[t_{12,34}(\xi),\bigwedgem{2}t_{4,5}(\zeta)]=t_{12,35}(\xi\zeta)$. Consequently, 
$$A_{I,J} \subseteq A_{K,L}\text{ for } \height(I, J) = \height(K, L) = 0.$$
\item Thirdly, let $\xi\in A_{12,13}$, then $[t_{12,13}(\xi),\bigwedgem{2}t_{1,4}(\zeta)]=t_{12,34}(-\xi\zeta)\in H$. Consider two commutators of the last transvection with $\bigwedgem{2}t_{4,1}(\zeta_1)$ and $\bigwedgem{2}t_{4,1}(-\zeta_1)$ respectively. We obtain that $t_{24,13}(\zeta_1^2\xi\zeta)\in H$ and also $t_{12,13}(-\xi\zeta\zeta_1)t_{24,34}(\zeta_1\xi\zeta)\in H$. Hence $t_{24,34}(\zeta_1\xi\zeta)\in H$. This means that 
$$A_{I,J} \subseteq A_{K,L}\text{ for any } \height(I, J) = \height(K, L) = 1.$$
\item Now, take any $\xi\in A_{12,23}$, then $[t_{12,23}(\xi),\bigwedgem{2}t_{4,2}(\zeta)]=t_{14,23}(-\zeta\xi)$. Thus
$$A_{I,J} \subseteq A_{K,L}\text{ for } \height(I, J)= 1, \height(K, L) = 0.$$
\item Finally, let $\xi\in A_{12,34}$. Like in (1), consider the commutator $t_{12,34}(\xi)$ with $\bigwedgem{2}t_{4,2}(\zeta)$. We obtain that $t_{14,23}(-2\xi\zeta^2)\in H$ and also $t_{14,34}(-\zeta\xi)t_{12,23}(-\xi\zeta)\in H$. By the same argument, we can provide these calculations with the transvection $t_{45,16}(\xi)$ and $\bigwedgem{2}t_{6,4}(\zeta_1)$. We get that $t_{56,14}(-2\zeta_1^2\xi)\in H$ and $t_{45,14}(\xi\zeta_1)t_{56,16}(\zeta_1\xi)\in H$. To finish the proof, it remains to commutate last two products. Then $t_{45,34}(-\xi^2\zeta_1\zeta)\in H$, or  
$$A_{I,J} \subseteq A_{K,L}\text{ for }\height(I, J)= 0, \height(K, L) = 1.$$
\end{enumerate}
\end{proof}

The following lemma is crucial for the rest. It is devoted to the alternative description of the relative elementary group.
\begin{lemma}\label{AlterRelatFor2}
Let $n\geq 6$. For any ideal $A\trianglelefteq R$, we have
$$\E(N,A)^{\bigwedgem{2}\E(n,R)}=\E(N,R,A),$$ where by definition $\E(N,R,A)=\E(N,A)^{\E(N,R)}.$
\end{lemma}
\begin{proof}
The inclusion $\leq$ is trivial. By Vaserstein--Suslin's lemma~\cite{VasSusSerre}, the group $\E(N,R,A)$, generated by elements of the form
$$z_{ij,hk}(\xi,\zeta) = z_{I,J}(\xi,\zeta) = t_{J,I}(\zeta)\,t_{I,J}(\xi)\,t_{J,I}(-\zeta), \;\xi \in A, \zeta \in R.$$
Hence to prove the reverse inclusion, it suffices to check that for any $\xi \in A$, $\zeta \in R$, the
matrix $z_{ij,hk}(\xi,\zeta)$ belongs to $F:=\E(N,A)^{\bigwedgem{2}\E(n,R)}$.

Let us consider two cases:
\begin{itemize}
\item Suppose that there exists one pair of the same indices. Without loss of generality, we can assume that $i=k$. Then this inclusion is obvious:
$$z_{ij,hi}(\xi,\zeta)={}^{t_{hi,ij}(\zeta)}t_{ij,hi}(\xi)={}^{\bigwedgem{2}t_{h,j}(\zeta)}t_{ij,hi}(\xi)\in F.$$
\item Thus, we are left with the inclusion $z_{ij,hk}(\xi,\zeta)\in F$ with different indices $i,\,j\,,h\,,k$. Firstly, we express
$t_{ij,hk}(\xi)$ as a commutator of elementary transvections:
$$z_{ij,hk}(\xi,\zeta)={}^{t_{hk,ij}(\zeta)}t_{ij,hk}(\xi)={}^{t_{hk,ij}(\zeta)}[t_{ij,jh}(\xi),t_{jh,hk}(1)].$$
Conjugating the arguments of the commutator by $t_{hk,ij}(\zeta)$, we get
$$z_{ij,hk}(\xi,\zeta)=[t_{ij,jh}(\xi)t_{hk,jh}(\zeta \xi),t_{jh,ij}(-\zeta)t_{jh,hk}(1)] =:[ab,cd].$$
Next, we decompose the right-hand side with the help of the formula 
$$[ab,cd] ={}^a[b,c]\cdot {}^{ac}[b,d]\cdot[a,c]\cdot {}^c[a,d],$$
and observe that the exponent $a$ belongs to $\E(N,A)$ so can be ignored. Now a direct calculation, based upon the Chevalley commutator formula, shows that
\begin{align*}
[b,c]&=[t_{hk,jh}(\zeta \xi),t_{jh,ij}(-\zeta)]=t_{hk,ij}(-\zeta^2 \xi) \in \E(N,A);\\
^c[b,d]&={}^{t_{jh,ij}(-\zeta)}[t_{hk,jh}(\zeta \xi),t_{jh,hk}(1)]=\\
&=t_{hk,ik}(-\xi\zeta^2(1+\xi\zeta))t_{jh,ik}(-\xi\zeta^2)\cdot \,{}^{\bigwedgem{2}t_{h,i}(\zeta)}[t_{hk,jh}(\xi\zeta),\bigwedgem{2}t_{j,k}(-1)];\\
[a,c]&=[t_{ij,jh}(\xi),t_{jh,ij}(-\zeta)]=[t_{ij,jh}(\xi),\bigwedgem{2}t_{h,i}(-\zeta)];\\
^c[a,d]&={}^{t_{jh,ij}(-\zeta)}[t_{ij,jh}(\xi),t_{jh,hk}(1)]=\\
&=t_{jh,ik}(\xi\zeta^2)t_{ij,ik}(-\xi\zeta)\cdot\,{}^{\bigwedgem{2}t_{h,i}(\zeta)}[t_{ij,jh}(\xi),\bigwedgem{2}t_{j,k}(-1)],
\end{align*}
where all factors on the right-hand side belong to $F$.
\end{itemize}
\end{proof}

\begin{remark}
The attentive reader can remark that these calculations are almost completely coincide with the calculations for the orthogonal and symplectic cases~\cite{VP-EOeven,VP-EOodd,VP-Ep}. In the special case $(n,m) = (4,2)$ calculations are the same due to the isomorphism $\bigwedgem{2}\E(4,R) \cong \EO(6,R)$. Amazingly that this argument proves similar proposition in the case of general exterior power (see Section~\ref{levelcomputation}, Lemma~\ref{AlterRelatForm}).
\end{remark}

\begin{corollary}\label{CorolOfL3}
Let $A$ be an arbitrary ideal of $R$; then
$$\bigwedgem{2}\E(n,R)\cdot\E(N,R,A)=\bigwedgem{2}\E(n,R)\cdot\E(N,A).$$
\end{corollary}

Summarizing the above two lemmas, we get the following crucial result.
\begin{theorem}[Level Computation]\label{LevelFor2}
Let $n\geq 6$, let $H$ be a subgroup in $\GL_{N}(R)$ containing $\bigwedgem{2}\E(n,R)$. Then there
exists a unique largest ideal $A\trianglelefteq R$ such that
$$\bigwedgem{2}\E(n,R)\cdot\E(N,R,A)\leq H.$$
Namely, if $t_{I,J}(\xi)\in H$ for some $I$ and $J$, then $\xi\in A$.
\end{theorem}

Lemma~\ref{AlterRelatFor2} asserts precisely that $\bigwedgem{2}\E(n,R)\cdot\E(N,R,A)$ is generated as a subgroup by transvections $\bigwedgem{2}t_{i,j}(\zeta)$, $\zeta\in R$, and by elementary transvections $t_{ij,hk}(\xi)$, $\xi\in A$, of the level $A$. As always, we assume that $n\geq 5$ and $2\in R^{*}$.

In the rest of the section, we prove the perfectness of the bottom restriction subgroup from the last Theorem.
\begin{lemma}\label{PerfectFor2}
Let $n\geq 6$. The group $\bigwedgem{2}\E(n,R)\cdot\E(N,R,A)$ is perfect for any ideal $A\trianglelefteq R$.
\end{lemma}
\begin{proof}
It suffices to verify that all generators of the group $\bigwedgem{2}\E(n,R)\cdot\E(N,R,A)$ lie in its commutator
subgroup, which we denote by $F$.
\begin{itemize}
\item For the transvections $\bigwedgem{2}t_{i,j}(\zeta)$ this follows from the Cauchy--Binet homomorphism:
$$\bigwedgem{2}t_{i,j}(\zeta)=\bigwedgem{2}([t_{i,h}(\zeta),t_{h,j}(1)])=\left[\bigwedgem{2}t_{i,h}(\zeta),\bigwedgem{2}t_{h,j}(1)\right].$$
\item On the other hand, for the transvections $t_{ij,hk}(\xi)$ with different indices $i,\,j\,,h\,,k$ and $\xi\in A$ this is obvious:
$$t_{ij,hk}(\xi)=[t_{ij,jh}(\xi),t_{jh,hk}(1)]=\left[t_{ij,jh}(\xi),\bigwedgem{2}t_{j,k}(\pm 1)\right]\in F.$$
If there exists a pair of the same indices, then 
$$t_{ij,hi}(\xi)=[t_{ij,ik}(\xi),t_{ik,hi}(1)]=\left[t_{ij,ik}(\xi),\bigwedgem{2}t_{k,h}(\pm 1)\right]\in F.$$ 
\end{itemize}
\end{proof}

\subsection{Exterior powers of elementary groups}\label{general}

In this section we generalize all previous statements to the case of an arbitrary exterior power functor. 

Let us define $m$-th exterior power of the $R$--module $R^n$ as follows. The basis of this module is all exterior products $e_{i_1}\wedge\ldots\wedge e_{i_m}$, where $1\leq i_1<\ldots<i_m\leq n$, such that $e_{\sigma(i_{1})}\wedge\ldots\wedge e_{\sigma(i_{m})} = \sign(\sigma)\, e_{i_1} \wedge \ldots \wedge e_{i_{m}}$ for any permutation $\sigma$ in the symmetric group $S_{m}$. Denote the exterior power of $R^n$ by $\bigwedgem{m}(R^n)$. Now, we define an action of the group $\GL_{n}(R)$ on basis elements of $\bigwedgem{m}(R^n)$ by the rule
$$g(e_{i_1}\wedge\ldots\wedge e_{i_m}):=(ge_{i_1})\wedge\ldots\wedge (ge_{i_m}).$$
By linearity we can extend the action to the whole module $\bigwedgem{m}(R^n)$. 

In other words, we can consider the Cauchy--Binet homomorphism $\bigwedgem{m}:\GL_{n}(R)\longrightarrow \GL_{\binom{n}{m}}\left(R\right)$ taking each matrix $x\in \GL_{n}(R)$ to the matrix $\bigwedgem{m}(x)\in \GL_{N}(R)$. Elements of $\bigwedgem{m}(x)\in \GL_{N}(R)$ are all $m$-order minors of the matrix $x$. More precisely, it is natural to index the entries of matrix $\bigwedgem{m}(x)$ by pairs of elements of the set $\bigwedgem{m}[n]$:
$$\left(\bigwedgem{m}(x)\right)_{I,J}=\left(\bigwedgem{m}(x)\right)_{(i_1,\dots,i_m),(j_1,\dots,j_m)}=M_{i_1,\dots,i_m}^{j_1,\dots,j_m}(x).$$
Note that $\bigwedgem{m}[n]$ is a set of weights for the group scheme $\bigwedgem{m}\GL_{n}$, while the whole set $[n]^{m}$ is a weight set for the group scheme $\GL_{\binom{n}{m}}$.

Then a group $\bigwedgem{m}\big(\GL_{n}(R)\big)$ by definition is an image of the general linear group under the Cauchy--Binet homomorphism.

We cannot but emphasize the difference between the groups\footnote{The same strict inclusions are still true with change $\SL$ to $\GL$.}
$$\bigwedgem{m} \big(\SL_{n}(R)\big) < \bigwedgem{m}\SL_{n}(R) < \SL_{\binom{n}{m}}(R).$$
The first one is obtained as a \textbf{group-theoretic} image of the [abstract] group $\SL_{n}(R)$ under the Cauchy--Binet homomorphism, while the second one is a group of $R$-points of the \textbf{categorical} image of the group scheme $\SL_{n}$ under the natural transformation corresponding to the previous Cauchy--Binet homomorphism. And as can be seen by dimension computation both last groups do not equal to the ambient group $\SL_{\binom{n}{m}}(R)$. We refer the reader to~\cite{VavPere} for more precise results about the difference between the last three groups.

We assume that $n \geq 2m$ due to an isomorphism $\bigwedgem{m}V^{*} \cong (\bigwedgem{\dim(V)-m}V)^{*}$ for an arbitrary free $R$-module $V$. As in Section~\ref{square}, $\bigwedgem{m}\E(n,R)$ is a normal subgroup of $\bigwedgem{m}(\GL_{n}(R))$ by Suslin's lemma. For the further computations we calculate an exterior power of an elementary transvection in the following proposition. The proof is a straightforward by the very definition of the [classical] Binet--Cauchy homomorphism.

\begin{proposition}\label{ImageOfTransvForm}
Let $t_{i,j}(\xi)$ be an elementary transvection, let $L$ be a naturally ordered subset of $[n]$ of cardinality $m-1$. Numbers $i$, $j$ does not belong to $L$. Denote by $\sign(L,i,j)$ the sign of the permutation $(L,i,j)$\footnote{ I.\,e., the permutation 
$\bigl(\begin{smallmatrix}
1 & 2 & \dots & n-1 & n & n+1\\
&& L & &  i  & j
\end{smallmatrix}\bigr)$.}. Then for $n\geq 3$, we have
$$\bigwedgem{m}t_{i,j}(\xi)=\prod\limits_{L\,\in\,\bigwedgem{m-1}\,[n\setminus \{i,j\}]} t_{L\cup i,L\cup j}(\sign(L,i,j)\xi) \eqno(m)$$
for any $1\leq i<j\leq n$.
\end{proposition}

It follows from the proposition that $\bigwedgem{m}t_{i,j}(\xi)\in \E^{\binom{n-2}{m-1}}(N,R)$, where by definition every element of the set $\E^M(N,R)$ is a product of $M$ or less elementary transvections.

\subsection{Elementary calculations technique}\label{levelgeneral}

In the case of general exterior power $m$, the calculations with elementary transvections are huge. This subsection is intended to impose theory upon these calculations.

In the next section we are interested in the exact description of the group $H\cap \E(N,R)$. The general strategy for such description is the following.
From the inclusion $\bigwedgem{m}\E(n,R) \leq H \textrm{ for any } t_{I,J}(\xi)\in H \textrm{ and } \bigwedgem{m}t_{i,j}(\zeta)\in\bigwedgem{m}\E(n,R)$
we see that $[t_{I,J}(\xi),  \bigwedgem{m}t_{i,j}(\zeta)]\in H$. Thereby if we commute the elementary transvection $t_{I,J}(\xi)\in H$ with elements from $\bigwedgem{m}\E(n,R)$ we obtain other transvections from the intersection $H\cap \E(N,R)$.

For example, from the formula $(m)$ it follows that
\begin{align*}
[t_{135, 124}(\xi), \bigwedgem{3}t_{7,6}(\zeta)] &= 1,\\
[t_{135, 124}(\xi), \bigwedgem{3}t_{4,6}(\zeta)] &= [t_{135, 124}(\xi), t_{124, 126}(\zeta)] = t_{135, 126}(\xi\zeta),\\
[t_{135, 124}(\xi), \bigwedgem{3}t_{4,3}(\zeta)] &= [t_{135, 124}(\xi), t_{124, 123}(\zeta)t_{145,135}(\zeta)] = t_{135, 123}(\xi\zeta)t_{145, 123}(\xi\zeta^2)t_{145, 124}(-\zeta\xi).
\end{align*}

Generalization of these calculations is the following result; proof goes by straightforward calculations.

\begin{proposition}\label{TypesOfComm}
Up to action of the symmetric group, there exist three types of commutator with a fixed transvection $t_{I,J}(\xi) \in \E\left(\binom{n}{m}, R\right)$:
\begin{enumerate}
\item $[t_{I,J}(\xi), \bigwedgem{m}t_{j,i}(\zeta)]=1$ if both $i\not\in I$ and $j \not \in J$ hold;
\item $[t_{I,J}(\xi), \bigwedgem{m}t_{j,i}(\zeta)] = t_{\tilde{I}, \tilde{J}}(\pm \zeta\xi)$ if either $i\in I$ or $j \in J$ holds. And then $\tilde{I} = I\backslash i \cup j$ or $\tilde{J} = J\backslash j \cup i$ respectively;
\item If both $i\in I$ and $j \in J$ hold; then we have more complicated equality:
$$[t_{I,J}(\xi), \bigwedgem{m}t_{j,i}(\zeta)] = t_{\tilde{I}, J}(\pm \zeta\xi)\cdot t_{I, \tilde{J}}(\pm \zeta\xi)\cdot t_{\tilde{I}, \tilde{J}}(\pm\zeta^2\xi).$$
\end{enumerate}
\end{proposition}
Note that last item is true whenever $I\setminus i\neq J\setminus j$, otherwise we obtain $[t_{I,J}(\xi),t_{J,I}(\pm\zeta)]$. This commutator cannot be presented more visualized than by the very definition.

The rule of the commutator calculations from the last proposition can be translated into the language of \textit{the weight diagrams}:\\[7mm]
\textbf{Weight diagrams tutorial.}
\begin{enumerate}
\item Let $G(A_{n-1},\blank)$ be a Chevalley--Demazure group scheme, and let $(I,J)\in\bigwedgem{m}[n]^2$ be a pair of different weights for the $m$-th exterior power of $G(A_{n-1},\blank)$. Consider any unipotent $x_\alpha(\xi)$ for a root $\alpha$ of the root system $A_{n-1}$, i.\,e., $x_\alpha(\xi)$ equals an elementary transvection $\bigwedgem{m}t_{i,j}(\xi) \in \bigwedgem{m}\E(n, R)$;
\item By $\Ar(\alpha)$ denote all paths on the weight diagram\footnote{Recall that we consider the representation with the highest weight $\varpi_{m}$.} of this representation corresponding to the root $\alpha$;
\item Then there exist three different items corresponding to the cases of Proposition~\ref{TypesOfComm}:
\begin{itemize}
\item sets of the initial and terminal vertices of paths from $\Ar(\alpha)$ do not contain the vertex $(I,J)$;
\item the vertex $(I,J)$ is initial or terminal for a one path from $\Ar(\alpha)$;
\item the vertex $(I,J)$ is simultaneously initial and terminal for some path\footnote{From root systems geometry any vertex can be initial or terminal for not more than one $\alpha$-path.} from $\Ar(\alpha)$.
\end{itemize}
\item Finally, consider the commutator of the transvection $t_{I,J}(\xi)$ and the element $\bigwedgem{m}t_{i,j}(\zeta)$. It equals a product of transvections. These transvections correspond to the paths from the previous item. The transvection arguments are monomials in $\xi$ and $\zeta$. Namely, in the second case the argument equals $\pm\xi\zeta$ and in the third case it equals $\pm\xi\zeta^2$.
\end{enumerate}

In Figure~\ref{Fig1}$(a)$ we present all three cases from item $(3)$ for $m=2$ and $\alpha = \alpha_{2}$:
\begin{itemize}
\item $(I,J) = (14,15)$, then $[t_{14, 15}(\xi), \bigwedgem{2}t_{2,3}(\zeta)] = 1$;
\item $(I,J) = (13,35)$, then $[t_{13, 35}(\xi), \bigwedgem{2}t_{2,3}(\zeta)] = t_{12, 35}(-\xi\zeta)$;
\item $(I,J) = (13,24)$, then $[t_{13, 24}(\xi), \bigwedgem{2}t_{2,3}(\zeta)] = t_{12, 24}(-\xi\zeta)t_{12,34}(\xi\zeta^2) t_{13,34}(\zeta\xi)$.
\end{itemize}

Similarly, for the case $m=3$ the elementary calculations can be seen directly from the Figure~\ref{Fig1}$(b)$.

\[
\xymatrix @+1.0pc {
{\overset{12}{\bullet}}\ar@{-}[r] \textbf{\ar@/^1pc/[r]^2}&{\overset{13}{\bullet}}\ar@{-}[r]\ar@{-}[d]&{\overset{14}{\bullet}}\ar@{-}[r]\ar@{-}[d]&{\overset{15}{\bullet}}\ar@{-}[d]\\
&{\overset{23}{\bullet}}\ar@{-}[r]&{\overset{24}{\bullet}}\ar@{-}[r]\ar@{-}[d] \textbf{\ar@/^1pc/[d]^2}&{\overset{25}{\bullet}}\ar@{-}[d] \textbf{\ar@/^1pc/[d]^2}\\
&&{\overset{34}{\bullet}}\ar@{-}[r]&{\overset{35}{\bullet}}\ar@{-}[d]\\
&&&{\overset{45}{\bullet}}\\
&&(a)}
\hspace{1cm}
\xymatrix @-2.0pc {
{\overset{123}{\bullet}}\ar@{-}[rrrr] &&&&{\overset{124}{\bullet}}\textbf{\ar@/^1pc/[rrrr]^4}\ar@{-}[rrrr]\ar@{-}[ddrrr]&&&&{\overset{125}{\bullet}}\ar@{-}[rrrr]\ar@{-}[ddrrr]&&&&{\overset{126}{\bullet}}\ar@{-}[ddrrr]\\
\\
&&&&&&&{\overset{134}{\bullet}}\ar@{-}[rrrr]\ar@{-}[ddd]\textbf{\ar@/^1pc/[rrrr]^4} &&&&{\overset{135}{\bullet}}\ar@{-}[rrrr]\ar@{-}[ddrrr]\ar@{-}[ddd]&&&&{\overset{136}{\bullet}}\ar@{-}[ddrrr]\ar@{-}[ddd]\\
\\
&&&&&&&&&&&&&&{\overset{145}{\bullet}}\ar@{-}[rrrr]\ar@{-}[ddd] &&&&{\overset{146}{\bullet}}\ar@{-}[ddrrr]\ar@{-}[ddd]\\
&&&&&&&{\overset{234}{\bullet}}\ar@{-}[rrrr]\textbf{\ar@/^1pc/[rrrr]^4} &&&&{\overset{235}{\bullet}}\ar@{-}[rrrr]\ar@{-}[ddrrr]&&&&{\overset{236}{\bullet}}\ar@{-}[ddrrr]\\
&&&&&&&&&&&&&&&&&&&&&{\overset{156}{\bullet}}\ar@{-}[ddd]\\
&&&&&&&&&&&&&&{\overset{245}{\bullet}}\ar@{-}[rrrr]\ar@{-}[ddd]\textbf{\ar@/^1pc/[ddd]^4} &&&&{\overset{246}{\bullet}}\ar@{-}[ddrrr]\ar@{-}[ddd]\textbf{\ar@/^1pc/[ddd]^4}\\
\\
&&&&&&&&&&&&&&&&&&&&&{\overset{256}{\bullet}}\ar@{-}[ddd]\textbf{\ar@/^1pc/[ddd]^4}\\
&&&&&&&&&&&&&&{\overset{345}{\bullet}}\ar@{-}[rrrr] &&&&{\overset{346}{\bullet}}\ar@{-}[ddrrr]\\
\\
&&&&&&&&&&&&&&&&&&&&&{\overset{356}{\bullet}}\ar@{-}[ddd]\\
\\
\\
&&&&&&&&&&&&&&&&&&&&&{\overset{456}{\bullet}}\\
&&&&&&&&&&&&&&(b)}
\]

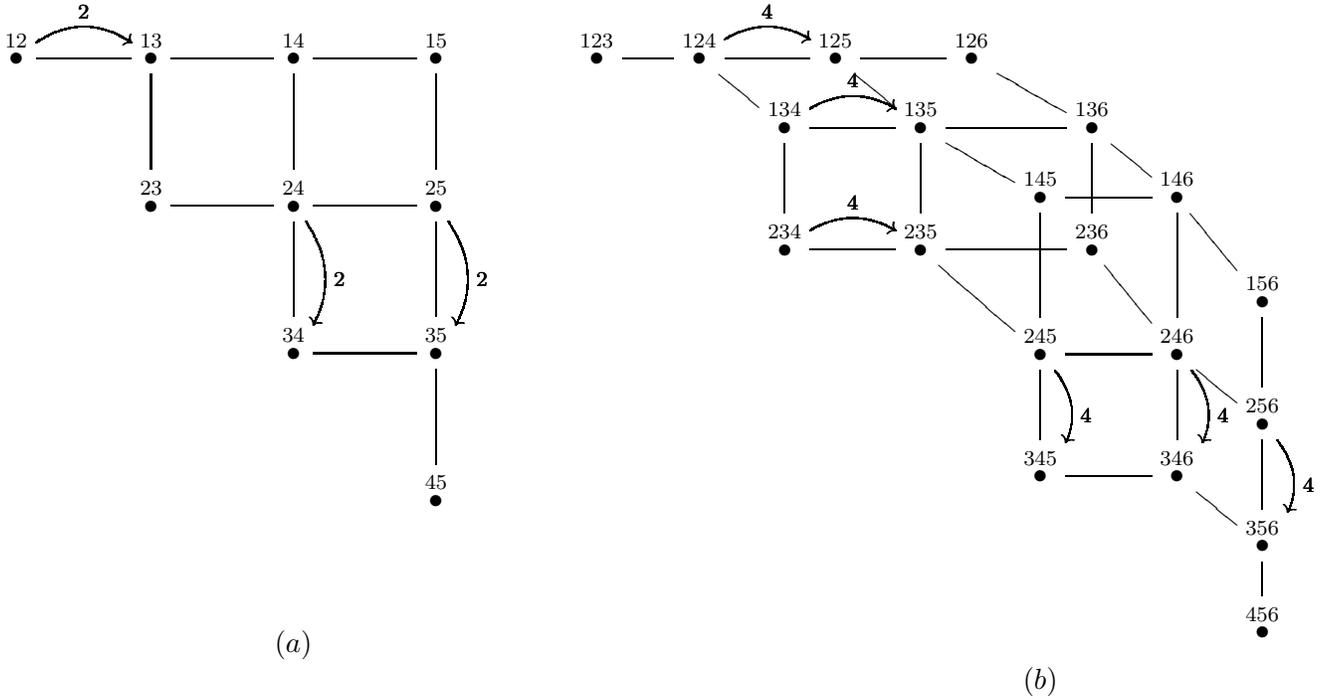
\captionof{figure}{Weight diagrams for $(a)$: $(A_4,\varpi_2)$, $\alpha = \alpha_{2}$ and $(b)$: $(A_5,\varpi_3)$, $\alpha = \alpha_{4}$}\label{Fig1}

\subsection{Level computation}\label{levelcomputation}
Generalize the notion of ideals $A_{I,J}$ for the $m$-th exterior power. $A_{I,J}:=\{\xi \in R \;|\; t_{I,J}(\xi) \in H \}$ for any different indices $I, J \in \bigwedgem{m}[n]$. Recall that the desired parametrization is given by an explicit juxtaposition for any overgroup $H$ its level, an ideal $A$ of the ring $R$. We compute this ideal $A$ in the present section.

As earlier, without loss of generality, we can assume that $n \geq 2m$. The first step toward the level description is the following observation.

\begin{proposition}\label{LevelsEqualGeneral}
If $|I \cap J|=|K\cap L|$, then sets $A_{I,J}$ and $A_{K,L}$ coincide. Actually, $A_{I,J}$ are ideals of $R$.
\end{proposition}
But first, we prove a weaker statement.
\begin{lemma} \label{LevelsEqual0}
Let $I, J, K, L$ be different elements of the set $\bigwedgem{m}[n]$, such that $|I \cap J|=|K\cap L| = 0$. If $n \geq 2m$, then the sets $A_{I,J}$ and $A_{K,L}$ coincide.
\end{lemma}
\begin{proof}[Proof of the lemma] First of all, the sets $A_{I, J}$ coincide when the set $I \cup J$ is fixed. This fact can be proved by the third type commutation due to Proposition~\ref{TypesOfComm} with $\zeta$ and $-\zeta$. If $\xi \in A_{I,J}$ we get a transvection $t_{I,J}(\xi) \in H$. Then the following two products belong to $H$:
\begin{align*}
[t_{I,J}(\xi), \bigwedgem{m}t_{j,i}(\zeta)] &= t_{\tilde{I}, J}(\pm \zeta\xi)\cdot t_{I, \tilde{J}}(\pm \zeta\xi)\cdot t_{\tilde{I}, \tilde{J}}(\pm\zeta^2\xi)\\
[t_{I,J}(\xi), \bigwedgem{m}t_{j,i}(-\zeta)] &= t_{\tilde{I}, J}(\mp \zeta\xi)\cdot t_{I, \tilde{J}}(\mp \zeta\xi)\cdot t_{\tilde{I}, \tilde{J}}(\pm \zeta^2\xi).
\end{align*}
This implies that the product of two factors on the right-hand side $t_{\tilde{I}, \tilde{J}}(\pm 2\zeta^2\xi)$ belongs to $H$.

The reader can easily prove that by the second type commutation we can change the set $I\cup J$. For example, a set $I_{1}\cup J_{1} = \{1,2,3,4,5,6\}$ can be replaced by a set $I_{2} \cup J_{2} = \{1,2,3,4,5,7\}$ as follows
$$[t_{123,456}(\xi), \bigwedgem{3}t_{6,7}(\zeta)] = t_{123, 457}(\xi\zeta).$$
\end{proof}

\begin{proof}[Proof of Proposition~$\ref{LevelsEqualGeneral}$]
Arguing as above, we see that the sets $A_{I,J}$ and $A_{K,L}$ coincide in the case $I\cap J = K\cap L$, where $n_{1} = n - |I\cap J| \geq 2\cdot m - 2 \cdot |I\cap J| = 2 \cdot m_{1}$.

In the general case, we can prove the statement by both the second and the third types commutation. Let us give an example of this calculation with replacing the set $I \cap J = \{1,2\}$ by the set $\{1,5\}$

Let $t_{123,124}(\xi) \in H$. So we have $[t_{123,124}(\xi),  \bigwedgem{3}t_{2,5}(\zeta) ] = t_{123,145}(-\xi\zeta) \in H$. We commute this transvection with the element $ \bigwedgem{3}t_{5,2}(\zeta_1)$. Then the transvection $t_{135,124}(-\zeta_1^2\xi\zeta)$ belongs to $H$ as well as the product $t_{123,124}(\xi\zeta\zeta_1)\cdot t_{135, 145}(-\zeta_1\xi\zeta)\in H$. From the last inclusion we can see that $t_{135, 145}(-\zeta_1\xi\zeta)\in H$ and $I\cap J = \{1,5\}$.

To prove that all $A_{I,J}$ are ideals in $R$ it sufficient to commute any elementary transvection with exterior transvections with $\zeta$ and $1$:
$$t_{I,J}(\xi\zeta)=[t_{I,J}(\xi),\bigwedgem{m}t_{j,i}(\zeta),\bigwedgem{m}t_{i,j}(\pm 1)]\in H.$$
\end{proof}

Let $t_{I,J}(\xi)$ be any elementary transvection. Define \textit{the height} of $t_{I,J}(\xi)$ (more abstractly, of the  pair $(I,J)$) as a cardinality of a set $I\cap J$:
$$\height(t_{I,J}(\xi))=\height(I,J)=|I\cap J|.$$
Then Proposition~\ref{LevelsEqualGeneral} can be rephrased as follows.
Sets $A_{I,J}$ and $A_{K,L}$ coincide for the same heights: $A_{I,J} = A_{K,L} =A_{|I\cap J|}$. Suppose that height of $(I,J)$ is larger than height of $(K,L)$, then using Proposition~\ref{TypesOfComm}, we get $A_{I,J} \subseteq A_{K,L}$

Summarizing the above arguments, we have the height grading:
$$A_{0} \supseteq A_{1} \supseteq A_{2} \supseteq \dots \supseteq A_{m-2} \supseteq A_{m-1}.$$

The following result proves the coincidence of all sets $\{A_{k}\}_{k=0\dots m-1}$.
\begin{prop}
The sets $A_{k}$ coincide for $n \geq 3m$. More accurately, the inverse inclusion $A_k\subseteq A_{k+1}$ takes place if $n\geq 3m-2k$.
\end{prop}
\begin{proof}
The statement can be proved by the double third type commutation as follows. Let $\xi\in A_k$, i.\,e., a transvection 
$t_{I,J}(\xi)\in H$ for $\height(t_{I,J}(\xi))=k$. By the third type commutation with a transvection $\bigwedgem{m}t_{j,i}(\zeta)$ we have $t_{\tilde{I}, J}(\pm \zeta\xi)\cdot t_{I, \tilde{J}}(\pm \zeta\xi)\in H$. Let us consider the analogous commutator with a specifically chosen transvections $t_{I_1,J_1}(\xi)\in H$ and $\bigwedgem{m}t_{j_1,i_1}(\zeta_1)$. The final step is to commute the last products.

The choice of transvections goes in the way that the final commutator (initially of the form $[ab,cd]$) has a form as simple as possible, e.\,g., equals an elementary transvection.

Let us give a particular example of such calculations for the case $m=4$. This calculation could be easily generalized. The first three steps below correspond to the inclusions $A_{0} \subseteq A_{1}$, $A_{1} \subseteq A_{2}$ and $A_{2} \subseteq A_{3}$ respectively. We emphasize that the idea of the proof of all three steps is \textbf{completely} identical. The difference has to do only with a choice of the appropriate indices. This choice is possible in a general case due to the condition $n \geq 3m$. We replace the numbers $10$, $11$, $12$ with the letters $\alpha$, $\beta$, $\gamma$ in the interest of readability.

\begin{enumerate}
\item Take any $\xi\in A_0$. Let us consider the commutator $[t_{1234,5678}(\xi), \bigwedgem{4}t_{8,4}(\zeta)]$. Then $$t_{1234,4567}(-\xi\zeta)\cdot t_{1238,5678}(-\zeta\xi)\in H.$$
Similarly,
$$t_{49\alpha\beta,1234}(\xi\zeta_1)\cdot t_{9\alpha\beta\gamma,123\gamma}(\zeta_1\xi)\in H$$
for the commutator $[t_{49\alpha\beta,123\gamma}(\xi), \bigwedgem{4}t_{\gamma,4}(\zeta_1)]$. It remains to commute the last two products. Therefore the transvection $t_{49\alpha\beta,4567}(\xi^2\zeta_1\zeta)$ belongs to $H$. As a result, $A_0\subseteq A_1$.
\item Take any $\xi\in A_1$. Let us consider the commutator $[t_{1234,1567}(\xi), \bigwedgem{4}t_{7,4}(\zeta)]$. Then $$t_{1234,1456}(\xi\zeta)\cdot t_{1237,1567}(-\zeta\xi)\in H.$$
and
$$t_{1489,1234}(\xi\zeta_1)\cdot t_{189\alpha,123\alpha}(-\zeta_1\xi)\in H$$
for the commutator $[t_{1489,123\alpha}(\xi), \bigwedgem{4}t_{\alpha,4}(\zeta_1)]$. Again, it remains to commute the last two products. Therefore the transvection $t_{1489,1456}(-\xi^2\zeta_1\zeta)$ lies in $H$. So, $A_1\subseteq A_2$.
\item Take any $\xi\in A_2$ and consider the commutator $[t_{1234,1256}(\xi), \bigwedgem{4}t_{6,4}(\zeta)]$. Then $$t_{1234,1245}(-\xi\zeta)\cdot t_{1236,1256}(-\zeta\xi)\in H.$$
and
$$t_{1248,1234}(\xi\zeta_1)\cdot t_{1278,1237}(-\zeta_1\xi)\in H$$
for the commutator $[t_{1248,1237}(\xi), \bigwedgem{4}t_{7,4}(\zeta_1)]$. After commuting the last two products we see that $t_{1248,1245}(\xi^2\zeta_1\zeta)\in H$. Then $A_2\subseteq A_3$.
\end{enumerate}
\end{proof}

As always, the set $A = A_{I,J}$ is called \textit{the level} of the overgroup $H$. For level computation we need an alternative description of the relative elementary group.

\begin{lemma}\label{AlterRelatForm}
Let $n\geq 3m$. For any ideal $A\trianglelefteq R$, we have
$$\E(N,A)^{\bigwedgem{m}\E(n,R)}=\E(N,R,A),$$
where by definition $\E(N,R,A)=\E(N,A)^{\E(N,R)}$.
\end{lemma}
\begin{proof}
Clearly, the left-hand side is contained in the right-hand side. The proof of the inverse inclusion goes by induction on a height of $(I,J)$. By Vaserstein--Suslin's lemma~\cite{VasSusSerre}, it suffices to check that for any $\xi \in A$, $\zeta \in R$, the
matrix $z_{I,J}(\xi,\zeta)$ belongs to $F:=\E(N,A)^{\bigwedgem{m}\E(n,R)}$.

In the base case $|I\cap J|=m-1$, the inclusion is obvious:
$$z_{I,J}(\xi,\zeta)\cdot t_{I,J}(-\xi)=[t_{J,I}(\zeta),t_{I,J}(\xi)]=\left[\bigwedgem{m}t_{j_1,i_1}(\zeta),t_{I,J}(\xi) \right]\in F.$$

Now, let us consider the general case $|I\cap J|=p$, i.\,e., $I=k_1\ldots k_{p}i_1\ldots i_{q}$ and $J=k_1\ldots k_pj_1\ldots j_q$. For the following calculations we need two more sets $V:=k_1\ldots k_pi_1\ldots i_{q-1}j_q$ and $W:=k_1\ldots k_pj_1\ldots j_{q-1}i_q$.\\
Firstly, we express $t_{I,J}(\xi)$ as a commutator of elementary transvections,
$$z_{I,J}(\xi,\zeta)=^{t_{J,I}(\zeta)}t_{I,J}(\xi)=^{t_{J,I}(\zeta)}[t_{I,V}(\xi),t_{V,J}(1)].$$
Conjugating the arguments of the commutator by $t_{J,I}(\zeta)$, we get
$$[t_{J,V}(\zeta\xi)t_{I,V}(\xi),t_{V,I}(-\zeta)t_{V,J}(1)]=:[ab,cd].$$
Next, we decompose the right-hand side with the help of the formula 
$$[ab,cd] ={}^a[b,c]\cdot {}^{ac}[b,d]\cdot[a,c]\cdot {}^c[a,d],$$
and observe that the exponent $a$ belongs to $\E(N,A)$ so can be ignored. Now a direct calculation, based upon the Chevalley commutator formula, shows that
\begin{align*}
[b,c]&=[t_{I,V}(\xi),t_{V,I}(-\zeta)]\in F \textrm{ (by the induction step for a height } m-1);\\
^c[b,d]&=^{t_{V,I}(-\zeta)}[t_{I,V}(\xi),t_{V,J}(1)]=t_{V,W}(\xi\zeta^2)t_{I,W}(-\xi\zeta)\cdot\; ^{\bigwedgem{m}t_{j_q,i_q}(-\zeta)}t_{I,J}(\xi);\\
[a,c]&=[t_{J,V}(\zeta\xi),t_{V,I}(-\zeta)]=t_{J,I}(-\zeta^2\xi);\\
^c[a,d]&=^{t_{V,I}(-\zeta)}[t_{J,V}(\zeta\xi),t_{V,J}(1)]=\\
=t_{J,W}(-\xi\zeta^2(1+\xi\zeta))&t_{V,W}(-\xi\zeta^2)\cdot\;^{\bigwedgem{m}t_{j_q,i_q}(-\zeta)}t_{J,V}(\xi\zeta)\cdot\;^{\bigwedgem{m}t_{j_q,i_q}(-\zeta)}z_{J,V}(-\zeta\xi,1)\in F,\\
\hspace{2cm}&\hspace{4cm}\textrm{ (by the induction step for a height }p+1)
\end{align*}
where all factors on the right-hand side belong to $F$.
\end{proof}

\begin{remark}
Since we do not use the coincidental elements of $I$ and $J$, we also can prove this Lemma by induction on  $|I\backslash J| = |J\backslash I| = 1/2 \cdot |I \triangle J|$. Then we can assume that $m$ is an arbitrarily large number (mentally, $m =\infty$).
\end{remark}

\begin{corollary}\label{CorolOfL6}
Let $A$ be an arbitrary ideal of the ring $R$; then
$$\bigwedgem{m}\E(n,R)\cdot\E(N,R,A)=\bigwedgem{m}\E(n,R)\cdot\E(N,A).$$
\end{corollary}

Summarizing the above two lemmas, we get the following crucial result.
\begin{thm}[Level Computation]
Let $n\geq 3m$, let $H$ be a subgroup in $\GL_{N}(R)$ containing $\bigwedgem{m}\E(n,R)$. Then there
exists a unique largest ideal $A\trianglelefteq R$ such that
$$\bigwedgem{m}\E(n,R)\cdot\E(N,R,A)\leq H.$$
Namely, if $t_{I,J}(\xi)\in H$ for some $I$ и $J$, then $\xi\in A$.
\end{thm}

\subsection{Normalizer of \texorpdfstring{$\E\bigwedgem{m}\E(n,R,A)$}{TEXT}}\label{normgeneral}

In this section we describe the normalizer of the bottom constraint for an overgroup $H$.
\begin{lemma}\label{PerfectForm}
Let $n\geq 3m$. The group $\E\bigwedgem{m}\E(n,R,A) := \bigwedgem{m}\E(n,R)\cdot\E(N,R,A)$ is perfect for any ideal $A\trianglelefteq R$.
\end{lemma}
\begin{proof}
As in Section~\ref{square}, it suffices to verify that all generators of the group $\bigwedgem{m}\E(n,R)\cdot\E(N,R,A)$ lie in its commutator subgroup, which will be denoted by $F$. The proof goes in two steps.
\begin{itemize}
\item For the transvections $\bigwedgem{m}t_{i,j}(\zeta)$ this follows from the Cauchy--Binet homomorphism:
$$\bigwedgem{m}t_{i,j}(\zeta)=\bigwedgem{m}([t_{i,h}(\zeta),t_{h,j}(1)])=\left[\bigwedgem{m}t_{i,h}(\zeta),\bigwedgem{m}t_{h,j}(1)\right].$$
\item For the linear transvections $t_{I,J}(\xi)$ this can be done as follows. Suppose that $I\cap J=K = k_1\ldots k_p$, where $0\leq p\leq m-1$, i.\,e., $I=k_1\ldots k_pi_1\ldots i_q$ and $J=k_1\ldots k_pj_1\ldots j_q$. As in Lemma~\ref{AlterRelatForm}, we define a set $V=k_1\ldots k_{p}j_1\ldots j_{q-1}i_q$. And then 
$$t_{I,J}(\xi)=[t_{I,V}(\xi),t_{V,J}(1)]=\left[t_{I,V}(\xi), \bigwedgem{m}t_{i_q,j_q}(\pm 1) \right],$$
so we get the required.
\end{itemize}
\end{proof}

Let, as above, $A\trianglelefteq R$, and let $R/A$ be the factor-ring of $R$ modulo $A$. Denote by $\rho_A: R\longrightarrow R/A$ the canonical projection sending $\lambda\in R$ to $\bar{\lambda}=\lambda+A\in R/I$. Applying the projection to all entries of a matrix, we get the reduction homomorphism 
$$\begin{array}{rcl}
\rho_{A}:\GL_{n}(R)&\longrightarrow& \GL_{n}(R/A)\\
a &\mapsto& \overline{a}=(\overline{a}_{i,j})
\end{array}$$
The kernel of the homomorphism $\rho_{A}$, $\GL_{n}(R,A)$, is called \textit{the principal congruence-subgroup in $\GL_{n}(R)$ of the level $A$}.

Now, let $\CC(n,R)$ be the center of the group $\GL_{n}(R)$, consisting of the scalar matrices
$\lambda e, \lambda \in R^{*}$. The full preimage of the center of $\GL_{n}(R/A)$, denoted by $\CC(n,R,A)$, is called the full congruence-subgroup of the level $A$. The group $\CC(n,R,A)$ consists of all matrices congruent to a scalar matrix modulo $A$. We further concentrate on a study of the full preimage of the group $\bigwedgem{m}\GL_{n}(R/A)$:
$$\CC\bigwedgem{m}\GL_{n}(R,A)=\rho_{A}^{-1}\left(\bigwedgem{m}\GL_{n}(R/A)\right).$$
A key point in reduction modulo an ideal is the following standard commutator formula, proved by Leonid Vaserstein~\cite{VasersteinGLn} and Zenon Borevich and Nikolai Vavilov~\cite{BV84}.
$$[\E(n,R),\CC(n,R,A)]=\E(n,R,A) \textrm{ for commutative ring }R \textrm{ and } n\geq 3.$$

Finally, we are ready to state the \textit{level reduction} result.
\begin{thm}
Let $n\geq 3m$. For any ideal $A\trianglelefteq R$ we have
$$N_{\GL_{N}(R)}(\E\bigwedgem{m}\E(n,R,A))=\CC\bigwedgem{m}\GL_{n}(R,A).$$
\end{thm}
\begin{proof}
In the proof by $N$ we mean $N_{\GL_{N}(R)}$.

Since $\E(N,R,A)$ and $\GL_{N}(R,A)$ are normal subgroups in $\GL_{N}(R)$, we see that
$$N(\underbrace{\E\bigwedgem{m}\E(n,R,A)}_{=\bigwedgem{m}\E(n,R)\E(N,R,A)})\leq N\left(\E\bigwedgem{m}\E(n,R,A)\GL_{N}(R,A)\right)=\CC\bigwedgem{m}\GL_{n}(R,A). \eqno(1)$$
Note that the last equality is due to normalizer functoriality:
$$N\left(\E\bigwedgem{m}\E(n,R,A)\GL_{N}(R,A)\right)=N\left(\rho_A^{-1}\left(\bigwedgem{m}\E(n,R/A)\right)\right)=\rho_A^{-1}\left(N \left(\bigwedgem{m}\E(n,R/A)\right)\right)=\rho_A^{-1}\left(\bigwedgem{m}\GL_{n}(R/A)\right).$$
In particular, using $(1)$, we get
$$\left[\CC\bigwedgem{m}\GL_{n}(R,A),\E\bigwedgem{m}\E(n,R,A)\right]\leq \E\bigwedgem{m}\E(n,R,A)\GL_{N}(R,A).\eqno(2)$$
On the other hand, it is completely clear that $\E\bigwedgem{m}\E(n,R,A)$ is normal in the right-hand side subgroup. Indeed, it is easy to prove the following stronger inclusion:
$$\left[\bigwedgem{m}\GL_{n}(R)\GL_{N}(R,A),\E\bigwedgem{m}\E(n,R,A)\right]\leq \E\bigwedgem{m}\E(n,R,A).\eqno(3)$$
To check this, we consider a commutator of the form
$$[xy,hg],\qquad x\in \bigwedgem{m}\GL_{n}(R), y\in \GL_{N}(R,A), h\in \bigwedgem{m}\E(n,R), g\in \E(N,R,A).$$
Then $[xy, hg]={}^x[y,h] \cdot [x,h] \cdot {}^h[xy,g]$. We need to prove that all factors on the right-hand side belong to $\E\bigwedgem{m}\E(n,R,A)$. Right away, the second factor lies in the group $\E\bigwedgem{m}\E(n,R,A)$. For the first commutator, we should consider the following inclusions:
$${}^{\bigwedgem{m}\GL_{n}(R)}\left[\GL_{N}(R,A),\bigwedgem{m}\E(n,R)\right]\leq \Bigl[
\underbrace{{}^{\bigwedgem{m}\GL_{n}(R)}\GL_{N}(R,A)}_{=\GL_{N}(R,A)},\underbrace{{}^{\bigwedgem{m}\GL_{n}(R)}\bigwedgem{m}\E(n,R)}_{=\bigwedgem{m}\E(n,R)}\Bigr]\leq \E\bigwedgem{m}\E(n,R,A).$$ 
The element $h\in \bigwedgem{m}\E(n,R)$, so we ignore it in conjugation. The third commutator lies in  $\E\bigwedgem{m}\E(n,R,A)$ due to the following inclusion. 
$$\left[\bigwedgem{m}\GL_{n}(R)\GL_{N}(R,A),\E(N,R,A)\right]\leq\left[\GL_{N}(R),\E(N,R,A)\right]=\E(N,R,A).$$
Now if we recall $(2)$ and $(3)$, we get
$$\left[\CC\bigwedgem{m}\GL_{n}(R,A),\E\bigwedgem{m}\E(n,R,A),\E\bigwedgem{m}\E(n,R,A)\right]\leq \E\bigwedgem{m}\E(n,R,A).\eqno(4)$$
To invoke the Hall–-Witt identity, we need a slightly more precise version of the last
inclusion:
$$\left[\left[\CC\bigwedgem{m}\GL_{n}(R,A),\E\bigwedgem{m}\E(n,R,A)\right],\left[\CC\bigwedgem{m}\GL_{n}(R,A),\E\bigwedgem{m}\E(n,R,A)\right]\right]\leq \E\bigwedgem{m}\E(n,R,A).\eqno(5)$$
Observe that by formula $(2)$ we have already checked that the left-hand side is generated by the commutators
of the form 
$$[uv,[z,y]], \text{ where } u,y\in \E\bigwedgem{m}\E(n,R,A), v\in \GL_N(R,A), z\in \CC\bigwedgem{m}\GL_n(R,A).$$
However,
$$[uv,[z,y]]={}^u[v,[z,y]]\cdot[u,[z,y]],$$
By formula $(4)$ the second commutator belongs to $\E\bigwedgem{m}\E(n,R,A)$, whereas by $(5)$ the first is an element of $\left[\GL_{N}(R,A),\E(N,R)\right]\leq \E(N,R,A)$. 

Now we are ready to finish the proof. By the previous lemma, the group $\E\bigwedgem{m}\E(n,R,A)$
is perfect, and thus, it suffices to show that $[z,[x,y]]\in \E\bigwedgem{m}\E(n,R,A)$ for all $x,y\in \E\bigwedgem{m}\E(n,R,A), z\in \CC\bigwedgem{m}\GL_{n}(R,A)$. Indeed, the Hall–-Witt identity yields
$$[z,[x,y]]={}^{xz}[[z^{-1},x^{-1}],y]\cdot{}^{xy}[[y^{-1},z],x^{-1}],$$
where the second commutator belongs to $\E\bigwedgem{m}\E(n,R,A)$ by $(4)$.
Removing the conjugation by $x\in \E\bigwedgem{m}\E(n,R,A)$ in the first commutator and carrying the conjugation by $z$ inside
the commutator, we see that it only remains to prove the relation $[[x^{-1},z],[z,y]y]\in \E\bigwedgem{m}\E(n,R,A)$. Indeed,
$$[[x^{-1},z],[z,y]y]=[[x^{-1},z],[z,y]]\cdot{}^{[z,y]}[[x^{-1},z],y],$$
where both commutators on the right--hand side belong to $\E\bigwedgem{m}\E(n,R,A)$ by formulas $(4)$ and $(5)$, and moreover,
the conjugating element $[z,y]$ in the second commutator is an element of the group $\E\bigwedgem{m}\E(n,R,A)\GL_{N}(R,A)$, and thus by $(3)$, normalizes $\E\bigwedgem{m}\E(n,R,A)$.
\end{proof}

\subsection{Exterior powers of \texorpdfstring{$\GL_{n}(R)$}{TEXT} as stabilizers}\label{stabgood}

As always, we assume that $n \geq 2m$ and $2\in R^{*}$.

Further on, by a stabilizer of a set of (symmetric or skew--symmetric) polynomials we mean a stabilizer of the induced polynomial $\GL_{n}$--representation. Note that a stabilizer of one polynomial $f(x)$ contains similarities, i.\,e., maps that preserve $f(x)$ up to a multiplier $\lambda$\footnote{This multiplier $\lambda(g)$ is a one-dimensional representation of the group $\GL_{n}(R)$, and thus, $\lambda(g)$ is a power of the determinant $\lambda=\det^{\otimes l}: g \mapsto \det^{l}(g)$. From the last, it follows that
this multiplier equals $1$ for $\SL_{n}$--form of the group.}: $g \circ f(x) = f(g \cdot x) = \lambda(g)\cdot f(x)$.

The goal of this section is to present the group $\bigwedgem{m}\GL_{n}(R)$ as a stabilizer group for some set of (symmetric or skew-symmetric) polynomials. 

The following classical theorem can be found in \cite[Chapter 2, Sections 5--7]{DieCarInvTheor}
\begin{proposition}\label{FormsGood}
The group\footnote{Equivalent formulation is the following: in decomposition of a module $\Sym^{k} \left(\bigwedgem{m} V^{n}\right)$ into irreducibles there are one-dimensional representations if and only if $n$ is divisible by $m$ and $m$ is even, the same is true for $\bigwedgem{k} \left(\bigwedgem{m} V^{n}\right)$ and odd $m$.} $\bigwedgem{m}\GL_{n}(R)$ has an invariant form only in the case $\frac{n}{m} \in \mathbb{N}$. In this case, this form is unique:
\begin{itemize}
\item $f_{n,m}(x) = \sum \sign(I_{1}, \dots, I_{\frac{n}{m}})\; x_{I_{1}}\cdot \dots \cdot x_{I_{\frac{n}{m}}}$ for even $m$;
\item $f_{n,m}(x) = \sum \sign(I_{1}, \dots, I_{\frac{n}{m}})\; x_{I_{1}}\wedge \dots \wedge x_{I_{\frac{n}{m}}}$ for odd $m$,
\end{itemize}
where the sums in the both cases range over all unordered partitions of the set $[n]$ into $m$-element subsets.
\end{proposition}

In the special case $\frac{n}{m} \in \mathbb{N}$, we can expect that the group $\bigwedgem{m}\GL_{n}(R)$ coincides with a stabilizer of the corresponding form. This is almost true and the following theorem gives the precise answer.

\begin{proposition}\label{StabAndGenerealPower}
We assume that $\frac{n}{m} \in \mathbb{N}$. Then 
\begin{enumerate}
\item in the case $n = 2m$ with $m\geq 3$ we have $\Stabi(f_{n,m}(x)) = \GO_{n}(R)$ or $\GSp_{n}(R)$ depending on a parity of $m$. So in this case $\bigwedgem{m}\GL_{n}(R)$ is a subgroup of orthogonal or symplectic group respectively;
\item for all other cases the group $\bigwedgem{m}\GL_{n}(R)$ coincides with a stabilizer group $\Stabi(f_{n,m}(x))$ of the form $f_{n,m}(x)$; 
\end{enumerate}
All stabilizers are considered as subgroups of a general linear group $\GL_{N}\left( R\right)$.
\end{proposition}

\begin{remark}
In the case $(n,m) = (4,2)$, the stabilizer group equals $\GO_{6}(R)$, but it also coincides with $\bigwedgem{2}\GL_{4}(R)$.
\end{remark}

\begin{proof}
Recall that we already know that in all cases the group $\bigwedgem{m}\GL_{n}(R)$ is a subgroup of $\Stabi(f_{n,m}(x))$ by Proposition~\ref{FormsGood}. So we have to analyze the reverse inclusion.

\textit{The idea of a proof} is to use the classification result of maximal subgroups in classical groups; we refer to the classical book~\cite[Table 1]{SeitzMaxSub}.

Using notation from this book, in our case $V$ is a free $\binom{n}{m}$--dimensional $R$--module, $X$ is the group $\bigwedgem{m}\GL_{n}(R)$. Then among overgroups of $\bigwedgem{m}\GL_{n}(R)$ can be only orthogonal or symplectic groups in a case $m\geq 3$. Moreover, for $m=2$ we know that $\bigwedgem{2}\GL_{n}(R)$ is maximal by the case $\MakeUppercase{\romannumeral 1}_{6}$ from Table 1 of \cite{SeitzMaxSub}. 

But in the case $\frac{n}{m} \geq 3$, by a direct calculation we can see that orthogonal or symplectic group is not contained in the corresponding stabilizer group.

On the contrary, in the case $n = 2m$ the form $f_{2m, m}(x)$ is equivalent to
$$f(y) = \sum\limits_{i = 1}^{m} y_{-i}\cdot y_{i}$$
for coordinates $y_{-m}, \dots, y_{-1}, y_{1}, \dots , y_{m}$. By definition, the stabilizer of the last one is  orthogonal or symplectic group depending on the parity of $m$.  
\end{proof}

For the case $\frac{n}{m} \not\in\mathbb{N}$, we assume that $n = lm +r$ for some $l,r \in \mathbb{N}$ and consider an ideal $I_{n,m}$, generated by forms $f_{m\cdot l, m}(x)$ for all choices of $m\cdot l$--element sets inside $[n]$. It is easy to prove that such ideals are representations of $\SL_n(R)$ (or $\GL_n(R)$) with the highest weights $\varpi_{m\cdot l}$, and its are \textbf{freely} generated by the forms. So rank equals $\binom{n}{m\cdot l}$. Then as a generalization of the first case, we have an analogous result presented in the following theorem.

\begin{proposition}\label{mNotDividesn}
In the case when $m$ does not divide $n$, the group $\bigwedgem{m}\GL_{n}(R)$ is a stabilizer of the ideal $I_{n,m}$.
\end{proposition}
\begin{proof}
In the same way from \cite[Table 1]{SeitzMaxSub}, we can see that there are no exceptional embeddings in these cases.
\end{proof}

Summing up, we have introduced (symmetric or skew--symmetric) polynomial ideals for every values of $(n,m)$ with a condition $\frac{n}{m}>2$. Further on, we denote this ideal by $I_{n,m}$.

\subsection{Normalizer Theorem}\label{transgood}

\begin{thm}
Let $R$ be a commutative ring, and let $n\geq 4$. Then
$$N_{\GL_{N}(R)}(\bigwedgem{m}\E(n,R)) = N_{\GL_{N}(R)}(\bigwedgem{m}\SL_{n}(R)) = \Tran_{\GL_{N}(R)}(\bigwedgem{m}\E(n,R)), \bigwedgem{m}\SL_{n}(R)) = \bigwedgem{m}\GL_{n}(R).$$
\end{thm}
\begin{proof} 
Clearly, $\bigwedgem{m}\GL_{n}(R) \leq N_{\GL_{N}(R)}(\bigwedgem{m}\SL_n(R))$ by the very definition of the extended Chevalley group. By \cite{PetrovStavrovaIsotropic}, we have $\bigwedgem{m}\GL_{n}(R) \leq N_{\GL_{N}(R)}(\bigwedgem{m}\E(n,R))$. On the other hand, both $N_{\GL_{N}(R)}(\bigwedgem{m}\E(n,R))$ and $N_{\GL_{N}(R)}(\bigwedgem{m}\SL_{n}(R))$ are obviously contained in $\Tran_{\GL_{N}(R)}(\bigwedgem{m}\E(n,R)), \bigwedgem{m}\SL_{n}(R))$. Thus, to finish the proof of the theorem, it suffices to check that $\Tran_{\GL_{N}(R)}(\bigwedgem{m}\E(n,R)), \bigwedgem{m}\SL_{n}(R))$ is contained in $\bigwedgem{m}\GL_{n}(R)$. This inclusion is proved in the case $\frac{n}{m}\neq 2$ in Proposition~\ref{TransInGeneralPower}, and in the case $n= 2m$ in Proposition~\ref{TransInGeneralPowerBad}.

\begin{proposition}\label{TransInGeneralPower}
In the case $\frac{n}{m}\neq 2$, we have the inclusion 
$$\Tran_{\GL_{N}(R)}(\bigwedgem{m}\E(n,R), \bigwedgem{m}\SL_{n}(R) )\leq \bigwedgem{m}\GL_{n}(R).$$
\end{proposition}
\begin{remark}
The bizarre, at first glance, choice of ideals for stabilizer representation for group $\bigwedgem{m}\GL_n(R)$ is entirely motivated (classically and presently) by the following proof.
\end{remark}
\begin{proof}[Proof of Proposition~$\ref{TransInGeneralPower}$]
We take any $g \in \Tran_{\GL_{N}(R)}(\bigwedgem{m}\E(n,R),\;\bigwedgem{m}\GL_{n}(R))$ and $ h \in \bigwedgem{m}\E(n, R)$, $a:= ghg^{-1}$. Then by $I_{n,m}^{g}$ we denote a set $\{f( g \circ x): \; f(x) \in I_{n,m}\}$, where $(g \circ x)_{J} = \sum\limits_{K} g_{JK}x_{K}$ is an action of the linear operator $g$ on coordinates $\{x_{I}\}_{I \in \bigwedgem{m}[n]}$. The set $I_{n,m}^{g}$ is a translation of the ideal $I_{n,m}$.

It is easy to prove that $I_{n,m}^{g}$ is an $\E(n,R)$--invariant ideal.

Also, we emphasize that the ideal $I_{n,m}$ was chosen\footnote{It is important, that $n$ is not less than product ``degree of polynomials'' $\times$ ``exterior power''.} as a unique ideal for the corresponding $k$ such that in the decomposition of $\Sym^{k}\left( \bigwedgem{m} V^{n} \right)$ or $\bigwedgem{k}\left( \bigwedgem{m} V^{n} \right)$ into irreducible modules
\begin{enumerate}
\item $I_{n,m}$ has multiplicity $1$, i.\,e., there are no isomorphic to $I_{n,m}$ modules in the decomposition;
\item dimension of the $I_{n,m}$ is the smallest in the decomposition\footnote{For the case $\frac{n}{m} \in \mathbb{N}$, this is a corollary of the uniqueness of stabilizing form. For the case $\frac{n}{m}\not \in \mathbb{N}$, this is a corollary of the uniqueness and hook length formula for dimension of the representation $V^{\lambda}$.}.
\end{enumerate}

The last ingredient we need is that the decomposition of $\Sym^{k}\left( \bigwedgem{m} V^{n} \right)$ or $\bigwedgem{k}\left( \bigwedgem{m} V^{n} \right)$ into $\E(n, R)$--invariant modules coincide with the one for $\GL_{n}(R)$. This is a corollary of the general idea that elementary subgroup $\E_G$ of any Chevalley--Demazure group scheme $G$ is dense. 

More precisely, \cite[Theorem 3.6]{DemGab} says that stabilizer of any irreducible $\SL_{n}(R)$ submodule in the decomposition is a closed subscheme in the general linear group. But element of the general linear group over a ring $R$ belongs to $R$-points of the closed subscheme if and only if it is true for every localization $R_{M}$ of a ring $R$ over maximal ideal $M$. Over any local ring $R_{M}$ elementary group coincides with $R_{M}$-points of Chevalley--Demazure group scheme $G$. From this follows that every irreducible submodule for the group $\SL_{n}(R)$ is irreducible for $\E(n,R)$ as well. Therefore the decompositions coincide.

Finally, we see that $I_{n,m}^{g}$ should be equal to $I_{n,m}$ because of the properties (1) and (2). So $g \in \bigwedgem{m}\GL_{n}(R)$, this case is done.
\end{proof}

This completes the proof of the theorem.
\end{proof}

\begin{remark}
The last proof in the case of 1-dimensional representation, i.\,e., of invariant form, is simpler. We present it below.

For the case $\frac{n}{m}\in \mathbb{N}$, let $g\in \Tran_{\GL_{N}(R)}(\bigwedgem{m}\E(n,R), \bigwedgem{m}\SL_{n}(R))$, $h\in \bigwedgem{m}\E(n,R)$. Then $a:=ghg^{-1}\in \bigwedgem{m}\SL_{n}(R)$, and thus $f_{n,m}(ax)=f_{n,m}(x)$ for all $x$. Substituting $gy$ for $x$, we get
$$f_{n,m}(gh y)=f_{n,m}(g y).$$
Next, let the map $F$ take each $v = \left(v_1,\ldots,v_{\binom{n}{m}}\right)$ to $f_{n,m}(g v)$. Then by our assumption $$F(h y)=F(y).$$
Hence, the form $F$ is invariant under the action of the elementary group $\bigwedgem{m}\E(n,R)$. Therefore, $F(y)=c\cdot f_{n,m}(y)$. Since this argument is true for $g^{-1}$, we obtain that the constant $c$ is invertible.
It follows that 
$$f_{n,m}(g y)=c\cdot f_{n,m}(y),\text{ i.\,e., }g\in \bigwedgem{m}\GL_{n}(R).$$
\end{remark}

\subsection{Stabilizers in the exceptional case \texorpdfstring{$n = 2m$}{TEXT}}\label{stabbad}

As we know from Proposition~\ref{StabAndGenerealPower}.1, in the special case $n = 2m$ the minimal-dimensional $\GL_{n}$-subrepresentation of $\Sym^{2}\left(\bigwedgem{m} V^{n} \right)$ (or $\bigwedgem{2}\left(\bigwedgem{m} V^{n} \right)$)  is stabilized by the whole orthogonal (or symplectic) group. So this representation cannot be used for stabilizer representation for the group $\bigwedgem{m} \GL_{n}(R)$.  

However, the decomposition of $\Sym^{2}\left(\bigwedgem{m} V^{n} \right)$ (or $\bigwedgem{2}\left(\bigwedgem{m} V^{n} \right)$) into irreducibles is classically known (see \cite[Exercise \MakeUppercase{\romannumeral 1}.8.9]{McDon}).
\begin{proposition} \label{Decomp}
We have the following decomposition into irreducible $\GL_{n}(R)$--modules
\begin{align*}
\Sym^{2}\left( \bigwedgem{m} V^{n} \right) &\cong \bigoplus_{l = 0, 2, 4 , \dots } V^{\lambda_{l}},\\
\bigwedgem{2}\left( \bigwedgem{m} V^{n} \right) &\cong \bigoplus_{l = 1, 3, \dots} V^{\lambda_{l}},
\end{align*}
where $\lambda_{l}$ is a Young diagram with two columns of size $n+l$ and $n-l$ and in both cases summations are going up to $n$.
\end{proposition}

The problem is that not one-dimensional subrepresentations are not free $R$-modules\footnote{Its combinatorial structure is complicated enough even over $\mathbb{C}$}, and , moreover, explicit generators are not known. So it is not clear how to use them separately for our goals. Nevertheless, we have the following stabilizer presentation.

\begin{proposition} \label{StabPresExc}
With fixed $m$ and $n$, for any nonempty set $I \subset \{2,3,4,5,\dots,n-1 \}$ stabilizer of the ideal generated by summands $\bigoplus_{i \in I} V^{\lambda_{i}}$ coincides with $\bigwedgem{m} \GL_{n}(R)$.  
\end{proposition}
\textit{The proof} goes exactly the same way as in Proposition~\ref{StabAndGenerealPower} using classification from~\cite[Table 1]{SeitzMaxSub}. The only thing worth to mention that the dimension of the representations $V^{\lambda_{l}}$ can be found by Hook Length Formula.

\medskip
So we get stabilization representation for all possible values of $(n,m)$, i.\,e., $m \geq 2, n \geq 4$. 
Before the last statement, we should recall some classical results about coordinate ring of the Grassmannian.

\begin{proposition}\label{DecompGrassmman}
A coordinate ring of a Grassmann variety $\Gr_{m}(\mathbb{C}^{n})$ over $\mathbb{C}$ as a $\GL_{n}(\mathbb{C})$-module isomorphic to infinite sum: 
$$\mathbb{C}\oplus( \mathbb{C}^{n})^{(1^{m})}\oplus ( \mathbb{C}^{n})^{(2\cdot 1^{m})} \oplus ( \mathbb{C}^{n})^{(3 \cdot 1^{m})} \oplus \dots.$$
\end{proposition}

That is, from the last theorem with a fact that Grassmannian as a scheme (over $\mathbb{Z}$) is an intersection of quadrics given by the generalized Pl\"ucker identities, we can deduce the following result.

\begin{corollary}\label{CorolOfthm9}
In notation of Proposition~\ref{Decomp}, the degree $2$ homogeneous component of the ideal generated by generalized Pl\"ucker identities is isomorphic as $\GL_{n}(R)$--module to  $\bigoplus\limits_{l = 2, 4 , \dots } V^{\lambda_{l}}.$ 

The same result with changing Pl\"ucker polynomials to the corresponding skew--symmetric polynomials is true.
\end{corollary}

Finally, we get the analog of Proposition~\ref{TransInGeneralPower} in the exceptional case. 

\begin{proposition}\label{TransInGeneralPowerBad}
We have the inclusion $\Tran_{\GL_{N}(R)}(\bigwedgem{m}\E(2m,R), \bigwedgem{m}\SL_{2m}(R) )\leq \bigwedgem{m}\GL_{2m}(R)$.
\end{proposition}
\begin{proof}
As always, $m \geq 3$. Then we can literally repeat the proof of Proposition~\ref{TransInGeneralPower}, with the exception of the following step. We use Corollary~\ref{CorolOfthm9} with  
$$
I =
\begin{cases}
\{1,3,5,\dots\} \text{ for even }m,\\
\{2,4,6,\dots\} \text{ for odd }m.
\end{cases}
$$
In other words, we take the full set of generalized Pl\"ucker identities in symmetric or exterior square whenever $m$ is odd or even in order not to take quadratic form stabilizing by orthogonal or symplectic group.

Then the ideal $I_{2m, m}^{g}$ might coincide with $I_{2m, m}$ or be subrepresentation (it is $\E(n, R)$ so is $\SL_{n}(R)$--invariant) of $I_{2m,m}$. But the last case corresponds to the case of Corollary~\ref{CorolOfthm9} with smaller set $I$. So an element $g$ in any case belongs to $\bigwedgem{m}\GL_{2m}(R)$.  
\end{proof}

\appendix
\section{Extended Chevalley--Demazure schemes: equivariant equations}\label{app}

In this Appendix, we state an algorithm for obtaining equations determining Chevalley--Demazure schemes. There are several attempts to get determining equations. For instance, Vladimir Popov in \cite{PopovEq} has found a complete list of equations for connected affine algebraic groups. 

However, we are mainly interested in equations agreed with algebraic group structure, \textit{equivariant equations}. Such sets of equations were presented for different Chevalley groups in series of the papers \cite{VP-EOeven,VP-EOodd,VP-Ep, VavLuzgE6,VavLuzgE7}; the \textbf{unifying principle} for these results we present below.

Let $\Phi$ be a reduced irreducible root system, let $G_{sc}(\Phi, \blank)$ be a corresponding simply connected Chevalle--Demazure scheme (over $\mathbb{Z}$), and $\rho$ be a representation of $G_{sc}(\Phi, \blank)$. 

We also consider the corresponding extended Chevalley--Demazure scheme $\overline{G}_{sc}(\Phi, \blank)$, which plays the same role with respect to $G_{sc}(\Phi, \blank)$ as the general linear group $\GL_n$ plays with respect to the special linear group $\SL_n$. Recall that extended groups in adjoint representation were constructed in the original paper~\cite{ChevalleySimpGroups} by Claude Chevalley. It is somewhat more tricky to construct extended groups in simply connected. The reason is unlike the adjoint case one have to increase the dimension of the maximal torus. A general construction was proposed by Stephen Berman and Robert V.~Moody in~\cite{BerMooExt}. 

The representation $\rho$ can be extended to the representation of the extended Chevalley--Demazure scheme. For simplicity, we denote this extended representation $\rho$ as well. 

By $\overline{G}_{\rho}(\Phi, \blank)$ we denote the scheme--theoretic image of the scheme $\overline{G}_{sc}(\Phi, \blank)$ in the representation $\rho$:
$$\rho: \overline{G}_{sc}(\Phi, \blank) \rightarrow \GL_{\dim \; \rho}(\blank).$$
Then we can state the general problem.

\begin{generalproblem}
For an extended Chevalle-Demazure scheme in representation $\rho$ to find the complete set of equations defining this scheme as an affine algebraic set.
\end{generalproblem} 

For an arbitrary ring $R$ let us consider symmetric algebra $\Sym_{R}(V_{\rho}) \cong R[x_{1}, \dots, x_{\dim \rho}]$. There is a natural action of $\GL(V_{\rho})$ on this algebra, so there is an action of the group $\overline{G}_{\rho}(\Phi, R)$. We decompose this polynomial space {into} irreducible components under the last action of the extended Chevalley--Demazure scheme
$$\Sym_{R}(V_{\rho}) \cong \bigoplus_{\lambda} V_{\rho}^{\lambda}.$$

The last means precisely that for each $\lambda$ there is a map 
$$\phi_{\rho, \lambda}:\;\overline{G}_{\rho}(\Phi, \blank) \rightarrow \GL(V_{\rho}^{\lambda}).$$

Let $k$ be a natural number such that a map 
$$\psi_{\rho, k}:\;\GL(V_{\rho}) \hookrightarrow \GL\left(\Sym^{k}(V_{\rho})\right)$$
is an embedding. By the classification of maximal overgroups of classical groups, except for the cases from~\cite[Table 1]{SeitzMaxSub}, we see that the image of $\overline{G}_{\rho}(\Phi, \blank)$ under the embedding $\psi_{\rho, k}$ is a stabilizer of a polynomial ideal $\langle V_{\rho}^{\lambda} \rangle$
$$\psi_{\rho, k}\left( \GL(V_{\rho}) \right) \cap \GL(V_{\rho}^{\lambda}) = \psi_{\rho, k} \big( \overline{G}_{\rho}(\Phi, \blank)\big).$$

The last equality gives us the desired list of equations for the scheme since both factors on the left-hand side can be {described} by polynomial equations. 

We note that the equations for $\GL(V_{\rho}^{\lambda}) \leq \GL(\Sym^{k}V_{\rho})$ could be very simple by the appropriate choice of a suitable subrepresentation. For instance, this is true for a one-dimensional space $V_{\rho}^{\lambda}$, which is nothing than a $\overline{G}_{\rho}(\Phi, \blank)$-invariant form\footnote{As always, we consider invariance of a form up to similarities.}.

Let us remark that this argument is true for the functor $\bigwedgem{k}$ instead of $\Sym^{k}$. We present several examples of this general idea below:
\begin{enumerate}
\item The Chevalley group $\SO_{n}$ in the natural representation has an invariant form (see~\cite{VP-EOeven,VP-EOodd})
$$x_{1}^2+\dots +x_{n}^2.$$
Let $P$ be a linear space generated by this form. The extended Chevalley group, $\GO_{n}$, is by definition a stabilizer of $P$. In other words,
$$\GO_{n}(R) = \{\;g \in \GL_{n}( R):\; g \circ (x_{1}^2+\dots +x_{n}^2) = \lambda(g)\cdot (x_{1}^2+\dots +x_{n}^2)\;\text{ for }\lambda(g) \in R^{*}\;\}.$$ 

\item The Chevalley group $\Sp_{2n}$ in the natural representation has an invariant form (see~\cite{VP-Ep}):
$$x_{n}y_{-n}+x_{n-1}y_{-n+1}+\dots+x_{1}y_{-1}-x_{-1}y_{1}-\dots-x_{-n}y_{n}.$$
And, again, the extended group is a stabilizer (up to similarities) of this form.
\item In~\cite{VavLuzgE6} Nikolai Vavilov and Alexander Luzgarev consider the Chevalley group of type $E_{6}$ in representation with the highest weight $\varpi_{1}$. They construct the invariant cubic form $Q$:
$$Q(x) = \sum_{\lambda, \mu, \nu} \sign(\lambda, \mu, \nu) x_{\lambda}x_{\mu}x_{\nu},$$
where the sum is taken over all unordered \textbf{triads} $\{\lambda, \mu,\nu\}\in \Theta_{0}$. The extended Chevalley group of type $E_{6}$ is a stabilizer of the linear space generated by $Q$.

\item Similarly, in~\cite{VavLuzgE7} Nikolai Vavilov and Alexander Luzgarev have considered the Chevalley group of type $E_{7}$ in representation with the highest weight $\varpi_{1}$. They construct a four-linear form $f$ and a bilinear form $h$ such that the extended Chevalley group is an intersection of the stabilizers of these forms.
\end{enumerate}

\bibliographystyle{unsrt}
\bibliography{english}

\end{document}